\documentclass[10pt,reqno]{amsart}
\usepackage[foot]{amsaddr}

\usepackage{hyperref}
\usepackage{latexsym,amsmath, bm}
\usepackage{enumerate}
\usepackage{amsfonts}
\usepackage{amssymb} 
\usepackage{geometry}
\usepackage{latexsym}
\usepackage{fixmath}
\usepackage{faktor}
\usepackage{mathtools}

\usepackage[T1]{fontenc}
\usepackage{fourier}
\usepackage{bbm}
\usepackage{color}

\usepackage{subcaption}

\usepackage{graphicx}
\usepackage{fullpage}

\usepackage[boxed, linesnumbered, noend, noline]{algorithm2e}
\SetKwInput{KwData}{Input}
\SetKwInput{KwResult}{Output}

\numberwithin{equation}{section}

\newcommand{\din}{d_{\mathrm{in}}}
\newcommand{\Min}{m_{\mathrm{in}}}
\newcommand{\Mout}{m_{\mathrm{out}}}
\newcommand{\Nin}{N_{\mathrm{in}}}
\newcommand{\Nout}{N_{\mathrm{out}}}
\newcommand{\dout}{d_{\mathrm{out}}}
\newcommand{\dsbm}{d_{\mathrm{sbm}}}
\newcommand{\agree}{\mathfrak{a}}

\newcommand{\sbm}{SBM}
\newcommand{\SBM}{\sbm}

\newcommand{\mom}[2]{#1^{(#2)}}

\newcommand{\algo}{\mathtt{AM}}
\newcommand{\ODE}{\mathtt{ODE}}

\renewcommand\subset{\subseteq}

\newcommand{\fT}{\mathfrak T}

\newcommand{\vd}{\vec d}

\newcommand{\vw}{\vec w}

\newcommand{\vX}{\vec X}
\newcommand{\vK}{\vec K}
\newcommand{\vf}{\vec f}
\newcommand{\vp}{\vec p}
\newcommand{\vE}{\vec E}
\newcommand{\vN}{\vec N}

\renewcommand{\epsilon}{\eps}

\newcommand\vU{\vec U}

\newcommand\nix{\,\cdot\,}
\newcommand\vV{\vec V}
\newcommand\vW{\vec W}

\newcommand\dd{{\mathrm d}}

\renewcommand{\vec}[1]{\boldsymbol{#1}}

\newcommand\SIGMA{\vec\sigma}

\newtheorem{definition}{Definition}[section]

\newtheorem{theorem}[definition]{Theorem}
\newtheorem{lemma}[definition]{Lemma}
\newtheorem{proposition}[definition]{Proposition}
\newtheorem{corollary}[definition]{Corollary}

\newtheorem{fact}[definition]{Fact}

\newtheorem{assumption}[definition]{Assumption}

\newcommand\fD{\mathfrak{D}}

\newcommand\fF{\mathfrak{F}}

\newcommand\fp{\mathfrak{p}}

\newcommand\cC{\mathcal{C}}

\newcommand\cT{\mathcal{T}}

\newcommand\cL{\mathcal{L}}
\newcommand\cM{\mathcal{M}}

\newcommand\cV{\mathcal{V}}

\newcommand\vu{\vec u}

\newcommand\vM{\vec M}

\newcommand\eul{\mathrm{e}}
\newcommand\eps{\varepsilon}

\newcommand{\vecone}{\mathbb{1}}

\newcommand{\set}[1]{\left\{#1\right\}}
\newcommand{\Po}{{\rm Po}}
\newcommand{\Bin}{{\rm Bin}}

\newcommand\bc[1]{\left({#1}\right)}
\newcommand\cbc[1]{\left\{{#1}\right\}}
\newcommand\bcfr[2]{\bc{\frac{#1}{#2}}}

\newcommand\brk[1]{\left\lbrack{#1}\right\rbrack}

\newcommand\norm[1]{\left\|{#1}\right\|}
\newcommand\abs[1]{\left|{#1}\right|}

\def\?#1{}
\makeatletter
\def\whp{w.h.p\@ifnextchar-{.}{\@ifnextchar.{.\?}{\@ifnextchar,{.}{\@ifnextchar){.}{\@ifnextchar:{.:\?}{.\ }}}}}}
\makeatother

\makeatletter
\def\Whp{W.h.p\@ifnextchar-{.}{\@ifnextchar.{.\?}{\@ifnextchar,{.}{\@ifnextchar){.}{\@ifnextchar:{.:\?}{.\ }}}}}}
\makeatother

\newcommand{\tensor}{\otimes}

\newcommand{\Erdos}{Erd\H{o}s}
\newcommand{\Renyi}{R\'enyi}

\newcommand\pr{\mathbb{P}} 
\renewcommand\Pr{\pr} 

\newcommand\Lem{Lemma}
\newcommand\Prop{Proposition}
\newcommand\Thm{Theorem}

\newcommand\Cor{Corollary}
\newcommand\Sec{Section}

\newcommand\id{\mathrm{id}}

 \def\G{{\vec G}}

\def\ex{{\mathbb E}}
\def\pr{{\mathbb P}}

\newcommand{\remove}[1]{}

\begin{document}
	
\title{Bad local minima exist in the stochastic block model}

\author{Amin Coja-Oghlan, Lena Krieg, Johannes Christian Lawnik, Olga~Scheftelowitsch}
\address{Amin Coja-Oghlan, {\tt amin.coja-oghlan@tu-dortmund.de}, TU Dortmund, Faculty of Computer Science and Faculty of Mathematics, 12 Otto-Hahn-St, Dortmund 44227, Germany.}
\address{Lena Krieg, {\tt lena.krieg@tu-dortmund.de}, TU Dortmund, Faculty of Computer Science, 12 Otto-Hahn-St, Dortmund 44227, Germany.}
\address{Johannes Christian Lawnik, {\tt johannes.lawnik@tu-dortmund.de}, TU Dortmund, Faculty of Computer Science, 12 Otto-Hahn-St, Dortmund 44227, Germany.}
\address{Olga Scheftelowitsch, {\tt olga.scheftelowitsch@tu-dortmund.de}, TU Dortmund, Faculty of Computer Science, 12 Otto-Hahn-St, Dortmund 44227, Germany.}

\begin{abstract}%
	We study the disassortative stochastic block model with three communities, a well-studied model of graph partitioning and Bayesian inference for which detailed predictions based on the cavity method exist [Decelle et al.\ (2011)].
	We provide strong evidence that for a part of the phase where efficient algorithms exist that approximately reconstruct the communities, inference based on maximum a posteriori (MAP) fails.
	In other words, we show that there exist modes of the posterior distribution that have a vanishing agreement with the ground truth.
	The proof is based on the analysis of a graph colouring algorithm from [Achlioptas and Moore (2003)].
	\hfill {\em MSc: 05C80,	62B10, 68R10}
\end{abstract}

\maketitle

\section{Introduction}\label{sec_intro}

\subsection{Background and motivation}
The {\em stochastic block model} (`\sbm') is a benchmark generative model of community detection in networks~\cite{AbbeSurvey,Moore}.
Over the past decade-and-a-half there has been an outburst of work on this model, much of it motivated by predictions derived via the cavity method from statistical mechanics~\cite{Decelle}.
Among other things, the predictions concerning the statistical and algorithmic thresholds of the model have largely been confirmed mathematically~\cite{abbe2015detection,CKPZ,mossel2013proof,Mossel,MSS}.
However, a point that has received only perfunctory attention thus far is the location of the modes of the posterior distribution.
Specifically, are the most likely configurations under the posterior aligned with the ground truth so that performing a maximum a posteriori (`MAP') inference is a viable strategy? 
Or might there exist modes that are totally uncorrelated with the ground truth despite the fact that non-trivial detection is possible by other means?

Abbe~\cite{AbbeSurvey} previously demonstrated that MAP fails on the stochastic block model with two communities for certain parameters, albeit for a trivial reason.
Indeed, because even past the statistical threshold the `giant component' of the two-community \SBM\ is still fairly small, one can construct MAP-optimal solutions that have a trivial agreement with the ground truth simply by trimming the small components.
The aim of the present paper is to investigate MAP on a well-studied variant of the \SBM\ where this circumstantial argument does not apply, namely the disassortative \SBM\ with three communities.
Perhaps surprisingly, the main result provides evidence that MAP is {\em not} generally a viable strategy even in this model, despite the fact that the largest component contains about 98\%\ of all vertices at the statistical threshold.

The precise definition of the model is as follows.%
	\footnote{We follow the notation and conventions from~\cite{CKPZ}.}
Let $V_n=\{v_1,\ldots,v_n\}$ be a set of vertices, let $d,\beta>0$ 
and set
\begin{align}\label{eqdindout}
	\din&=\frac{3d\exp(-\beta)}{2+\exp(-\beta)},&\dout&=\frac{3d}{2+\exp(-\beta)}.
\end{align}
The {\em ground truth} is a uniformly random colouring $\SIGMA^*:V_n\to\{1,2,3\}$ of $V_n$ with three colours.
Given $\SIGMA^*$ we generate a random graph $\G^*$ by connecting any two vertices $v,w\in V_n$, $v\neq w$, independently with probability
\begin{align}\label{eqpvw}
	p_{v,w}=1\wedge\frac1n\brk{\vecone\{\SIGMA^*(v)=\SIGMA^*(w)\}\din+\vecone\{\SIGMA^*(v)\neq\SIGMA^*(w)\}\dout}.
\end{align}
Thus, the average degree of $\G^*$ equals $d$.
Further, there are three communities $\SIGMA^{*\,-1}(1),\SIGMA^{*\,-1}(2),\SIGMA^{*\,-1}(3)$ and the choice~\eqref{eqpvw} of the edge probabilities ensures that intra-com\-munity edges are less likely than edges across communities, the more so the larger $\beta$.

The objective is to reconstruct the ground truth $\SIGMA^*$ given $\G^*$ and the parameters $d,\beta$.
Clearly, this amounts to a Bayes-optimal inference task.
Given $\beta>0$ this tasks gets easier as $d$ increases, i.e., as more edges are observed.
To be precise, the {\em Kesten-Stigum threshold}
	\begin{equation}\label{eqdsbm}
	\dsbm(\beta)=\bcfr{2+\exp(-\beta)}{1-\exp(-\beta)}^2.
	\end{equation}
is conjectured to mark the threshold value of $d$ beyond which a non-trivial reconstruction of the ground truth is possible by means of an efficient algorithm~\cite{Decelle}.
Formally, since we can at best expect to reconstruct $\SIGMA^*$ modulo a permutation of the colours, we define the {\em agreement} of a colouring $\sigma\to\{1,2,3\}$ with the ground truth as
\begin{align} \label{eqagree}
	\agree(\sigma)= -\frac12+\frac3{2n}\max_{\kappa\in\mathbb S_q}\sum_{v \in V_n} \vecone \{\SIGMA^*(v) = \kappa\circ\sigma(v)   \}.
\end{align}
Then a strictly positive value of $\agree(\sigma)=\Omega(1)$ indicates that $\sigma$ renders a non-trivial approximation to $\SIGMA^*$.

Abbe and Sandon~\cite{abbe2015detection} contributed an algorithm that for $d>\dsbm(\beta)$ outputs a colouring $\SIGMA_{\mathrm{alg}}:V_n\to\{1,2,3\}$ such that $\agree(\SIGMA_{\mathrm{alg}})=\Omega(1)$ with high probability (`\whp').
Conversely, for $d<\dsbm(\beta)$ the \SBM\ $\G^*$ and the \Erdos-\Renyi\ random graph $\G(n,d/n)$ are conjectured to be mutually contiguous, which implies that a non-trivial reconstruction of the ground truth is statistically impossible in this regime~\cite{Decelle}.

\subsection{Results}
There are two generic approaches to reconstructing the ground truth given the data in any Bayesian inference problem.
The first is to draw a sample from the posterior distribution, i.e., $\pr\brk{\SIGMA^*=\nix\mid\G^*}$ in the case at hand.
The second is to find a maximiser $\sigma$ of the posterior $\pr\brk{\SIGMA^*=\nix\mid\G^*}$; this latter strategy goes by the name of {\em maximum a posteriori} (`MAP') inference.
The computational appeal of MAP is that we `just' need to solve an optimisation problem, rather than a potentially (even) less tractable sampling problem.
While it is not currently known if either of these strategies can be implemented efficiently on the \sbm\ and although for $d>\dsbm(\beta)$ a different efficient algorithm is known to find $\sigma$ with $\agree(\sigma)=\Omega(1)$ \whp, it is of fundamental interest to understand the performance of these two schemes~\cite{AbbeSurvey}.

It is not difficult to see that the statistically optimal reconstruction algorithm is to sample from the posterior~\cite{CEJKK,CKPZ,Decelle}.
By contrast, the main result of the present paper provides evidence that MAP does not generally produce a $\sigma$ with $\agree(\sigma)=\Omega(1)$ for $d>\dsbm$ \whp.
To be precise, consider the {\em loss function}
\begin{align}\label{eqloss1}
	L_{\G^*}(\sigma)=\frac1n\log\frac{\max_{\tau\in\{1,2,3\}^{V_n}}\pr\brk{\SIGMA^*=\tau\mid\G^*}}{\pr\brk{\SIGMA^*=\sigma\mid\G^*}}\geq0.
\end{align}
We will show, modulo an analytic assumption, that for certain parameters $\beta>0$, $d>\dsbm(\beta)$ \whp\ there exist $\sigma$ with $\agree(\sigma)=o(1)$ such that 
$L_{\G^*}(\sigma)=o(1)$, while $L_{\G^*}(\SIGMA^*)=\Omega(1)$.
In other words, $\sigma$ is a `bad' minimum of the loss function that has no discernible agreement with the ground truth.
In effect, minimising the loss function (viz.\ MAP inference) is not generally a good strategy.
Instead, the decisive advantage of sampling from the posterior over MAP is that the former implicitly takes advantage of the extensive `crater' of the loss function around the ground truth $\SIGMA^*$.
In other words, sampling from the posterior takes into consideration the entropy of the posterior, not just the value of the loss function.

To prove the existence of the `bad' minimiser of the loss function we analyse a combinatorial 3-colouring algorithm on the stochastic block model. 
We consider this algorithm because if a 3-colouring (with colour classes of equal sizes) exists then such a colouring is a MAP-solution.
The algorithm was originally proposed by Achlioptas and Moore~\cite{AchMoore3col} in order to prove that \Erdos-\Renyi\ random graphs of average degrees up to $4.03$ have chromatic number three with high probability. 
The analysis of the algorithm relies on Wormald's method of differential equations~\cite{Wormald}.
Unfortunately, no analytic solution of the system of differential equations is known.
Instead, we rely on a numerical solution.
Because of this, we state the properties of the solution as an {\em assumption} to our main theorem.
We will discuss the merit of this assumption in the appendix.

The system $\ODE(d,\alpha,\Delta)$ of ordinary differential equations comes with two real parameters $d,\alpha>0$ and one integer parameter $\Delta>0$.
The functions of the system are $u_0(t),\ldots,u_{\Delta}(t)$ and $w_0(t),\ldots,w_{\Delta}(t)$.
Let $\Po_{\leq\Delta}(d)$ denote a Poisson variable with mean $d$ conditioned on an outcome of less than or equal to $\Delta$.
Furthermore, let
\begin{align}\label{eqphi}
	\phi&=d^{-1}\sum_{i>\Delta}i\pr\brk{\Po(d)=i}.
\end{align}
Then the initial conditions read
\begin{align}\label{eqODEinitial}
	u_\ell(0)&=\sum_{j=1}^{\Delta-\ell}\pr\brk{\Po_{\leq\Delta}(d)=\ell+j}\pr\brk{\Bin(\ell+j,\phi)=j},&
	w_\ell(0)&=\pr\brk{\Po_{\leq\Delta}(d)=\ell}(1-\phi)^\ell&&(0\leq\ell\leq\Delta).
\end{align}
Furthermore, the system of differential equations reads
\begin{align}\label{eqODE}
	\frac{\dd w_\ell}{\dd t}&=-\frac{\kappa\ell w_\ell}{\sum_{i=1}^\Delta i(u_i+w_i)},\qquad
	\frac{\dd u_\ell}{\dd t}=\kappa\frac{\vecone\{\ell<\Delta\}(\ell+1)(w_{\ell+1}+u_{\ell+1}/3)-\ell u_\ell}{\sum_{i=1}^\Delta i(u_i+w_i)}-\frac{\ell^\alpha u_\ell}{\sum_{i=1}^\Delta i^\alpha u_i}\qquad\mbox{where}\\
\kappa&=\frac{3\sum_{i=1}^\Delta i(u_i+w_i)}{\sum_{i=1}^\Delta 3iw_i+(5i-2i^2)u_i}\cdot\frac{\sum_{i=1}^\Delta i^{\alpha+1}u_i}{\sum_{i=1}^\Delta i^{\alpha}u_i}
\label{eqkappa}
\end{align}

\begin{assumption}\label{assumption_ode}
	Given $d>0$ we assume that there exists $\Delta_0>0$, $\alpha\geq1$, $\xi>0$ and $t^*\in(0,1)$ such that for any $\Delta>\Delta_0$ the system~\eqref{eqODEinitial}--\eqref{eqkappa} possesses a solution on the interval $[0,t^*+\xi]$ that satisfies 
	\begin{align}\label{eqODEpos}
		u_\ell(t),w_\ell(t)&>\xi&&\mbox{for all $0\leq t\leq t^*+\xi$, $0\leq\ell\leq\Delta$},\\
		\label{eqODEend}
		\sum_{\ell=0}^\Delta\ell(\ell-2)(u_\ell(t^*)+w_\ell(t^*))&<-\xi,&&\mbox{and}\\
		\frac{\sum_{i=1}^\Delta i(i-1)u_i(t)}{\sum_{i=1}^\Delta i(w_i(t)+u_i(t))}&<\frac32-\xi&&\mbox{for all }0\leq t\leq t^*+\xi.
\label{eqODElambda}
	\end{align}
\end{assumption}

The main result of this paper shows that under Assumption~\ref{assumption_ode} the loss function of stochastic block model $\G^*$ \whp\ possesses a bad minimum $\sigma$ of vanishing loss $o(1)$ and vanishing agreement with the ground truth, while the actual ground truth $\SIGMA^*$ has loss $\Omega(1)$.

\begin{theorem}\label{thm_main}
	Let $d>0$ be such that Assumption~\ref{assumption_ode} holds and let $\beta>0$.
	Then \whp\ there exists $\sigma:V_n\to\{1,2,3\}$ such that
	\begin{align*}
		\agree(\sigma)&=o(1)&\mbox{ and }&&L_{\G^*}(\sigma)&=o(1),&\mbox{while}&&
		L_{\G^*}(\SIGMA^*)&=\frac{d\beta+o(1)}{4\eul^{\beta}+2}.
	\end{align*}
\end{theorem}

Numerical and experimental evidence suggests that Assumption~\ref{assumption_ode} is satisfied for $d$ up to at least $4.03$ (see Appendix~\ref{apx:num_res}).
Since the threshold $\dsbm(\beta)$ from~\eqref{eqdsbm} tends to $4$ as $\beta\to\infty$, this demonstrates that for large enough $\beta$ there exist $d>\dsbm(\beta)$ such that \Thm~\ref{thm_main} `bites'; specifically, for $\beta=6$ we have $\dsbm(\beta)\approx4.0299$.
Hence, for $d=4.03$ and $\beta>6$ a sample $\SIGMA$ from the posterior $\pr\brk{\SIGMA^*=\nix\mid\G^*}$ \whp\ has extensive agreement $\agree(\SIGMA)=\Omega(1)$ with the ground truth, while there exist $\sigma$ with loss $o(1)$ such that $\agree(\sigma)=o(1)$.

The proof of \Thm~\ref{thm_main} actually shows that under Assumption~\ref{assumption_ode} a {\em strict} minimiser of the loss function, i.e., an actual mode of the posterior, with vanishing agreement exists with a probability that remains bounded away from zero as $n\to\infty$.

\begin{corollary}\label{cor_main}
	Let $d>0$ be such that Assumption~\ref{assumption_ode} holds and let $\beta>0$.
	Then with probability $\Omega(1)$ there exists $\sigma:V_n\to\{1,2,3\}$ such that
\begin{align*}
		\agree(\sigma)&=o(1)&\mbox{ and }&&L_{\G^*}(\sigma)&=0,&\mbox{while}&&
		L_{\G^*}(\SIGMA^*)&=\frac{d\beta+o(1)}{4\eul^{\beta}+2}.
	\end{align*}
\end{corollary}

\subsection{Further related work}\label{sec_related}
The stochastic block model was originally suggested by Holland, Laskey and Leinhardt~\cite{Holland}.
Although the model had been investigated earlier in various contexts, its modern study began with the heuristic work of Decelle,  Krzakala, Moore, Zdeborov\'a~\cite{Decelle}.
Based on the cavity method from statistical physics~\cite{MM}, they predicted thresholds for the non-trivial reconstruction of the planted partition.
Specifically, they obtained information-theoretic and algorithmic thresholds.
The former disregard algorithmic efficiency, while the latter ask for a polynomial time reconstruction algorithm.
A substantial amount of rigorous work has by now been dedicated to proving these predictions.
The surveys~\cite{Abbe,Moore} provide an excellent summary of this body of work.%
	\footnote{We point out that the performance of MAP inference improves once we go beyond the regime of the {\em sparse} \sbm, i.e., the case of bounded average degree, which is the focus of the present work.  Specifically, once $d=\Omega(\log n)$ the MAP solution can be expected to align with the ground truth; see~\cite{Abbe}.}
Two more recent important contributions \cite{MSS,Sly23} deal with the precise reconstruction thresholds in the assortative and disassortative case.

A well-known open question that is only indirectly related to the stochastic block model is to determine the chromatic number of a sparse \Erdos-\Renyi\ random graph. 
In fact, among the numerous open problems posed in their seminal paper~\cite{ER60}, this is the only that that still remains unsolved.
The objective of the aforementioned work of Achlioptas and Moore~\cite{AchMoore3col} was to provide a better lower bound on the 3-colourability threshold of the \Erdos-\Renyi\ random graph.
Their result (average degree $4.03$) still stands as the best current lower bound.
For general numbers of colours Achlioptas and Naor~\cite{AchNaor} obtained lower bounds on the colourability threshold via the (non-constructive) second moment method; see also~\cite{Danny}.
The best current upper bounds on the colourability thresholds can be found in~\cite{catherine} (the bound for three colours is $4.697$, matching a prediction from~\cite{pnas}).

An important paper by Alon and Kahale~\cite{AlonKahale} expressly deals with the 3-community \sbm\ with $\beta=\infty$, also known as the planted 3-colouring model.
Alon and Kahale provide an efficient algorithm that produces a 3-colouring of this random graph under the assumption that $d$ exceeds a sufficiently large (implicit) constant.
An open question raised in their paper asks whether such an algorithm exists for {\em all} $d$.
The proof of \Cor~\ref{cor_main} provides a (very) partial answer: under Assumption~\ref{assumption_ode} for $d<4.03$ a 3-colouring can be found in polynomial time with at least a non-vanishing probability $\Omega(1)$.

Finally, this paper has partly been inspired by the experimental/heuristic work of Liu, Papailiopoulos and Achlioptas~\cite{LPA} on the performance of stochastic gradient descent.
The main point of that paper is to present examples where the stochastic gradient descent algorithm gets stuck in local minima of the loss function that generalise poorly.
In a sense, the present paper establishes at least the {\em existence} of a corresponding bad local minimum on the \sbm\ in which a reasonable colouring algorithm gets stuck (under Assumption~\ref{assumption_ode}).
Indeed, if we imagine that the edges of the \SBM\ are presented in an online fashion, then the partition $\sigma$ from \Thm~\ref{thm_main} with $\agree(\sigma)=o(1)$ also generalises poorly as every new edge has about a $1/3$ probability of landing squarely within one of the partition classes.

\section{Proof strategy}\label{sec_outline}

\noindent
The proof of \Thm~\ref{thm_main} is algorithmic.
Specifically, we are going to analyse the aforementioned 3-colouring algorithm that we call $\algo$ from Achlioptas and Moore~\cite{AchMoore3col}, originally devised to prove that \Erdos-\Renyi\ random graphs of average degree $4.03$ are $3$-colourable \whp\
Accordingly, Achlioptas and Moore analysed $\algo$ on the \Erdos-\Renyi\ model by means of the method of differential equations.
The ensuing system of differential equations is precisely the one displayed in~\eqref{eqODEinitial}--\eqref{eqkappa}. 

To prove \Thm~\ref{thm_main} we will adapt the analysis of Achlioptas and Moore to the \SBM.
Although this may sound pretty straightforward, the devil is in the detail.
Specifically, the differential equations that describe the evolution of the algorithm on the stochastic block model come in terms of a {\em much} larger number of variables, namely basically $12(\Delta+1)^3$ variables by comparison to the $2(\Delta+1)$ variables $w_\ell$, $u_\ell$ from~\eqref{eqODE}.
Also the combinatorial (branching process) argument from which these differential equations are derived seems at first glance far more intricate than the corresponding analysis on the \Erdos-\Renyi\ model.
Yet we will show that the number of variables can be reduced and that the branching process can be simplified, so that ultimately the very same differential equations that `drive' the $\algo$ algorithm on the \Erdos-\Renyi\ graph also describe its evolution on the \SBM.

\subsection{The $\algo$ algorithm}\label{sec_AM}

The $\algo$ algorithm attempts to find a proper 3-colouring of its input graph $G$, i.e., a colouring of the vertices such that any two adjacent vertices receive distinct colours.
In every iteration of its main loop the algorithm assigns one of the three possible colours $\{1,2,3\}$ to a vertex for good; thus, there is no backtracking.
For every vertex $v$ that has not been coloured yet the algorithm maintains a list $\cL_v$ of colours that can potentially be assigned to $v$.
An initial set of lists $(\cL_v)_{v\in V(G)}$ is part of the input.
We are going to assume that for at least a few vertices $v$ the initial list $\cL_v$ contains only two colours.
For a vertex $v$ we denote by $\partial v=\partial_Gv$ its set of neighbours in the graph $G$.

Given $G$ and $(\cL_v)_{v\in V(G)}$, the algorithm sets about its task as follows.
So long as there exist uncolured vertices $v$ whose list $\cL_v$ only contains a single colour, $\algo$ assigns that colour to $v$.
Subsequently the algorithm removes $v$ from the graph (along with any incident edges) and also removes the colour assigned to $v$ from the lists of any neighbours $u\in\partial v$.
If no such `forced' vertices exists, the algorithm looks for a vertex $v$ with $|\cL_v|=2$.
Among these vertices the algorithm chooses one randomly with a probability that depends on the remaining number of uncoloured neighbours of $v$.
To be precise, the probability of selecting a given $v$ is proportional to the number of remaining uncoloured neighbours raised to the power of $\alpha\geq1$.
Think of $\alpha$ as a `large' number so that the algorithm prioritises vertices of high degree.
The chosen vertex $v$ is then coloured with a uniformly random colour from its list $\cL_v$.
Subsequently we remove $v$ from $G$ and delete the colour assigned to $v$ from the lists $\cL_u$ of all neighbours $u\in\partial v$.
Finally, if the algorithm runs out of vertices with one or two available colours, it stops and returns the resulting partial colouring $\sigma$.

Since every time the algorithm assigns a colour to a vertex $v$ it removes that colour from the lists of all neighbours $u\in\partial v$ of $v$, it can happen that some vertices $u$ run out of colours entirely, i.e., that $\cL_u$ runs empty.
We call these vertices $u$ {\em bad}.
Clearly, whenever a bad vertex occurs the algorithm fails at its task of constructing a proper 3-colouring of $G$.
Yet the analysis of $\algo$ will show that under Assumption~\ref{assumption_ode} the expected number of bad vertices on the \SBM\ is bounded, and thus negligible towards the proof of \Thm~\ref{thm_main}.
In fact, to prove \Cor~\ref{cor_main} we will argue that with probability $\Omega(1)$ no bad vertices occur.

A second issue is that $\algo$ fails colour {\em all} vertices of $G$ unless every connected component of the input graph contains a vertex $v$ whose initial list satisfies $|\cL_v|<3$.
This is because the {\tt while}-loop stops if no vertices $v$ with $|\cL_v|\in\{1,2\}$ remain.
Hence, we will need to mop up the remainder of the graph by different means.
As we will see, on the \SBM\ under Assumption~\ref{assumption_ode} this is easy because the remainder is a sub-critical random graph \whp, i.e., it decomposes into small and (mostly) acyclic components.
The pseudocode of $\algo$ is displayed below as Algorithm~\ref{alg_AM}.

\IncMargin{1em}
\begin{algorithm}[h!]
	\KwData{A graph $G$ and initial lists $(\cL_v)_{v\in G}$}
 \KwResult{a partial colouring $\sigma$}
 \While{$\exists v\in V(G):0<|\cL_v|<3$}{

	\If{$\forall v\in V(G):|\cL_v|\geq2$}{pick $v\in V(G)$ with $|\cL_v|=2$ randomly with probability
		$$p_{G,\alpha}(v)=\frac{|\partial v|^\alpha}{\sum_{w\in V(G):|\cL_w|=2}|\partial w|^\alpha}$$
	and then select $\sigma_v\in\cL_v$ uniformly; if $|\partial w|=0$ for all $w\in V(G)$, then select $v\in V(G)$ uniformly at random.}
	 \Else
	 {pick a uniformly random $v\in V(G)$ with $|\cL_v|=1$ and let $\sigma_v\in\cL_v$}
\For{$u\in\partial v$}{remove colour $\sigma_v$ from $\cL_u$}
	delete vertex $v$ from $G$\\
	}
	\Return $\sigma$
   \caption{$\algo$.}\label{alg_AM}
\end{algorithm}
\DecMargin{1em}

\subsection{High-degree vertices}\label{sec_run_algo}
To prove \Thm~\ref{thm_main} we are going to analyse $\algo$  on the \SBM\ $\G^*$ via the method of differential equations from~\cite{Wormald}.
Achlioptas and Moore used the same method to analyse the algorithm on the \Erdos-\Renyi\ model, which is identical to the \SBM\ with $\beta=0$.
Following their approach, we will track the evolution of the numbers of vertices of each degree and colour list $\cL_v$ as $\algo$ executes.
In addition, we will need to take the planted colour $\SIGMA^*(v)$ into account.

However, as also pointed out in~\cite{AchMoore3col}, an inherent limitation of the differential equations method is that it can only track a bounded number of variables.
Yet the maximum degree of $\G^*$ diverges as $n\to\infty$.
To deal with this issue, we follow Achlioptas and Moore by first colouring and deleting all vertices whose degree exceeds an a priori cutoff $\Delta$.
Thus, obtain $\G^*_\Delta$ by deleting from $\G^*$ all vertices of degree greater than $\Delta$.

We need to extract a bit of information about $\G^*_\Delta$.
First and foremost, what is the number of vertices of each planted colour that remain?
Formally, let $\vN_s$ be the number of vertices $v$ of $\G^*_\Delta$ with $\SIGMA^*(v)=s$.
Furthermore, let $\vN_{s,d_1,d_2,d_3}$ be the number of vertices $v$ with $\SIGMA^*(v)=s$ that have precisely $d_r$ neighbours $w$ such that $\SIGMA^*(w)=r$ in $\G^*_\Delta$ for $r=1,2,3$.
Routine calculations evince the following.

\begin{fact}\label{prop_Delta}
	\Whp\ for $s=1,2,3$ we have
	\begin{align}\label{eqNs}
		\vN_s&=\frac n3\pr\brk{\Po(d)\leq\Delta}+O(\sqrt n\log n).
	\end{align}
	Furthermore, \whp\ for all $d_1,d_2,d_3\geq0$ with $d_1+d_2+d_3\leq\Delta$ we have
	\begin{align}\label{eqNsd1d2d3}
		\vN_{s,d_1,d_2,d_3}&=\frac n3\pr\brk{\Po(d)=d_1+d_2+d_3}\binom{d_1+d_2+d_3}{d_1,d_2,d_3}\frac{\exp(-\beta d_s)}{(2+\exp(-\beta))^{d_1+d_2+d_3}}+O(\sqrt n\log n).
	\end{align}
\end{fact}

Of course, we also need to colour the vertices of degree greater than $\Delta$.
Hence, let $V^\#$ contain all of vertices of degree greater than $\Delta$ as well as all vertices of degree less than or equal to $\Delta$ that have a neighbour of degree greater than $\Delta$.
Moreover, let $\G^\#$ be the subgraph of $\G^*$ with vertex set $V^\#$ in which two vertices $v,w\in V^\#$ are adjacent iff
\begin{itemize}
	\item $v,w$ are adjacent in $\G^*$, and
	\item at least one of $v,w$ has degree greater than $\Delta$.
\end{itemize}
In addition, let $V^\#_\Delta$ be the set of all vertices in $V^\#$ that have degree less than or equal to $\Delta$.
The following proposition summarises an \SBM-adaptation of the strategy of Achliopatas and Moore for dealing with high-degree vertices.

\begin{proposition}\label{prop_high}
	For any $d>0$ there exists $\Delta_0=\Delta_0(d)>0$ such that for all $\beta\geq0$ and $\Delta\geq\Delta_0$ \whp\ there exists $\SIGMA^\# : V^\#\to\{1,2,3\}$ with the following properties.
	\begin{enumerate}[(i)]
		\item For all but $o(n)$ edges $vw$ of $\G^\#$ we have $\SIGMA^\#(v)\neq\SIGMA^\#(w)$.
		\item For every vertex $v\in V^\#_\Delta$ there exists a colour $c^\#(v)\in\{1,2,3\}$ such that $\SIGMA^\#(w)=c^\#(v)$ for all $w\in\partial v\cap V^\#$.
		\item For any two colours $c,c'\in\{1,2,3\}$ we have
			\begin{align*}
				\abs{\cbc{v\in V^\#:\SIGMA^*(v)=c,\,\SIGMA^\#(v)=c'}}&=\frac{1+o(1)}9|V^\#|,\\
				\abs{\cbc{v\in V^\#_\Delta:\SIGMA^*(v)=c,\,c^\#(v)=c'}}&=\frac{1+o(1)}9|V^\#_\Delta|.
			\end{align*}
	\end{enumerate}
	Furthermore, $\SIGMA^\#$ is a proper 3-colouring of $\G^\#$ with probability $\Omega(1)$. 
\end{proposition}

Thus, we can assign colours to the high-degree vertices and their neighbours in such a way that only $o(n)$ monochromatic edges ensue.
Furthermore, all neighbours of a vertex of degree at most $\Delta$ have the same colour.
Additionally, the colours that $\SIGMA^\#$ assigns to high degree vertices, as well as the `forbidden' colours $c^\#(v)$ of low degree vertices, are uncorrelated with the planted colours.
The proof of \Prop~\ref{prop_high} can be found in \Sec~\ref{sec_prop_high}.

\subsection{The principle of deferred decisions}\label{sec_deferred}
We now proceed to run $\algo$ on $\G^*_\Delta$. 
The initial lists of colours are determined by the colouring of the high degree vertices.
Thus, in the notation of \Prop~\ref{prop_high}, for all $v\in V^\#_\Delta$ we set the initial list to $\cL_v(0)=\{1,2,3\}\setminus\{c^\#(v)\}$.
For all other vertices $u$ of degree at most $\Delta$ we let $\cL_v(0)=\{1,2,3\}$.

Given the input $\G^*_\Delta,(\cL_v(0))_{v\in V(\G^*_\Delta)}$, how will the execution of $\algo$ unfold?
Routine arguments show that $\G^*_\Delta$ consists of one `giant' connected component that contains about 98\%\ of all vertices, and a linear number $\Omega(n)$ of smaller components of maximum size $O(\log n)$.
In fact, most of the small components contain only a bounded number of vertices, and all but a bounded number of components are acyclic \whp\
In effect, since the initial lists contain $\Omega(n)$ vertices with list size two \whp, $\algo$ will \whp\ encounter the giant component within the first $O(1)$ executions of its main loop.
As the algorithm proceeds to colour the giant of $\G^*_\Delta$, the algorithm removes vertices of this component one by one.
As a consequence, the giant component will ultimately fall apart and decompose into a number of mostly acyclic pieces of bounded size.
Our strategy is to model the execution of $\algo$ via differential equations until this decomposition occurs; for once the giant component shatters into pieces, the remaining components will be easy to colour.
Specifically, we are going to show that up until the dissolution of the giant component few bad vertices (whose colour list runs empty) occur \whp\

To carry this strategy out let us call an {\em epoch} of the algorithm the steps performed following an execution of Step~3 until just before the next execution of Step~3, or until the termination of the algorithm.
Thus, an epoch consists of assigning a colour to a vertex $v$ with $|\cL_v|=2$, i.e., with two available colours, and the subsequent `forced' colourings of vertices with only one available colour, until none are left.
Furthermore, by the $t$-th epoch we mean the epoch that starts with the $t$-th execution of Step~3.
Due to the aforementioned component structure of $\G^*$, \whp\ by the time $\algo$ stops we will thus be left with a few small and (mostly) acyclic uncoloured components where all vertices still have all three colours available.
This remainder is easy to colour.

Let $V(t)$ be the set of vertices of $\G^*_\Delta$ that have not yet been coloured at the end of the $t$-th epoch.
Moreover, let $\G^*_{\Delta}(t)$ be the graph that remains at the end of the $t$-th epoch.
Similarly, for a vertex $v\in V(t)$ let $\cL_v(t)$ be the colour list of $v$ at the end of epoch $t$.
Further, for $v\in V(t)$ let $\vd_{v,s}(t)$ be the number of neighbours $w$ of $v$ in $\G^*_{\Delta}(t)$ with $\SIGMA^*(w)=s$.
Finally, let $\fD(t)$ be the $\sigma$-algebra generated by $(\vd_{v,s}(t))_{v\in V(t),s\in\{1,2,3\}}$, $(\cL_v(t))_{v\in V(t)}$ and $\SIGMA^*$.
The following proposition provides that $\G^*_{\Delta}(t)$ is uniformly random given $\fD(t)$.

\begin{fact}\label{prop_deferred}
	Given $\fD(t)$, $\G^*_{\Delta}(t)$ is uniformly distributed on the set of all graphs $G$ on $V(t)$ such that every $v\in V(t)$ has precisely $\vd_{v,s}(t)$ neighbours in the set $V_s(t)=\{w\in V(t):\SIGMA^*(w)=s\}$.
\end{fact}
\begin{proof}
	This is an instance of the `principle of deferred decisions'.
	Specifically, the decisions of $\algo$ (which vertex to colour next and what colour it receives) depend only on the degrees of the vertices and their colour lists.
	Furthermore, once a vertex receives a colour, it gets deleted from the graph.
	Therefore, the only information that gets carried over is implicit in the degrees and the colour lists.
\end{proof}

\subsection{Expectation management}\label{sec_condex}
We now define a number of random variables that, in light of Fact~\ref{prop_deferred}, suffice to track the execution of $\algo$.
Subsequently we will see that the evolution of these random variables in $t$ can be described in terms of a system of ordinary differential equations.

Let $\cT_\Delta=\{(s,d_1,d_2,d_3):s\in\{1,2,3\},d_1,d_2,d_3\geq0,\,d_1+d_2+d_3\leq\Delta\}$.
Thinking of $s$ as the planted colour $\SIGMA^*(v)$ and of $d_j$ as the number of neighbours $w$ with planted colour $\SIGMA^*(w)=j$, we refer to the elements $(s,d_1,d_2,d_3)\in\cT_\Delta$ as {\em vertex types}.
For $\theta=(s,d_1,d_2,d_3)\in\cT_\Delta$ with $d_1+d_2+d_3<\Delta$ and $\chi\in\{1,2,3\}$ let 
\begin{align*}
	\theta^{+\chi}&=(s,d_1+\vecone\{\chi=1\},d_2+\vecone\{\chi=2\},d_3+\vecone\{\chi=3\}).
\end{align*}
In words, this is the type obtained from $\theta$ by allowing for one more neighbour $w$ with $\SIGMA^*(w)=\chi$.

For a list $\cL\subset\{1,2,3\}$ of size $|\cL|\geq2$ and a type $\theta=(s,d_1,d_2,d_3)\in\cT_\Delta$ let $\vV_{\cL,\theta}(t)$ be the number of vertices $v\in V(t)$ such that $\SIGMA^*(v)=s$, $\cL_v(t)=\cL$ and $\vd_{v,\chi}(t)=d_\chi$ for $\chi=1,2,3$.
Further, let
\begin{align*}
	\vW_\theta(t)&=\vV_{\{1,2,3\},\theta}(t)
\end{align*}
be the number of `white' vertices of type $\theta$ that still have all three colours available at time $t$.
Additionally, let
\begin{align*}
	\vU_\theta(t)&=\sum_{\cL\subset\{1,2,3\}:|\cL|=2}\vV_{\cL,\theta}(t)
\end{align*}
be the total number of vertices of type $\theta$ with two available colours.
Finally, for a colour $s\in\{1,2,3\}$ let
\begin{align*}
	\vU_{s,\theta}(t)&=\vV_{\{1,2,3\}\setminus\{s\},\theta}(t)
\end{align*}
be the number of vertices of type $\theta$ that do {\em not} have colour $s$ available.
For notational convenience, for any tuple $(s,d_1,d_2,d_3)$ with $d_1+d_2+d_3>\Delta$ we define $\vV_{\cL,\theta}(t)=\vW_\theta(t)=\vU_\theta(t)=\vU_{s,\theta}(t)=0$.

We proceed to derive a closed system of equations for the conditional expectations of the random variables $\vW_\theta(t+1),\vU_{s,\theta}(t+1)$ given $\fD_t$.
To this end, for $\theta=(s,d_1,d_2,d_3)$, $\theta'=(s',d_1',d_2',d_3')$ and $c,c'\in\{1,2,3\}$ and $c''\in\{1,2,3\}\setminus\{c,c'\}$ define 
\begin{align}
	\vp_{c, \theta}(t) = & \frac{(d_1+d_2+d_3)^\alpha\sum_{\chi\in\{1,2,3\}\setminus\{c\}} \vU_{\chi,\theta}(t)}{2\sum_{\theta'\in\cT_{\Delta}}(d_1'+d_2'+d_3')^\alpha\vU_{\theta'}(t)}\label{eqpctheta},\\
	\vE_{s, s'}(t) =& \sum_{v\in V(t):\SIGMA^*(v)=s} \vd_{c'}(v) = \sum_{v\in V(t):\SIGMA^*(v)=s'}\vd_s(v)\label{eqEss'},\\
	\vM_{c', \theta', c, \theta}(t) =& \vecone\{c\neq c'\}\bc{d_{s'} (d_{s}'+1) \vU_{c'',{\theta'}^{+s}}(t)}/\vE_{s,s'}(t)\label{eqMmatrix},\\  
	\vec\kappa_{c, s, s'}(t) = & \sum_{d_1,d_2,d_3=0}^\Delta \vecone\{(s,d_1,d_2,d_3)\in\cT_\Delta\}d_{s'}\left(\left(\id -\vM(t) \right)^{-1}\vec p(t)\right)_{(c, (s, d_1 , d_2, d_3))} \label{eqvkappacss'},\\
	\vec\kappa_{s, s'}(t) = & \sum_c \vec\kappa_{c, s,s'}(t).\label{eqkappass'}
\end{align}
We think of $\vM(t)=(\vM_{c',\theta',c,\theta}(t))_{(c',\theta),(c,\theta)\in\{1,2,3\}\times\cT_\Delta}$ as a $(3|\cT_\Delta|)\times(3|\cT_\Delta)$-matrix and of $\vp(t)=(\vp_{c,\theta}(t))_{c,\theta}$ as a row vector of dimension $3|\cT_\Delta|$.
Furthermore, in~\eqref{eqvkappacss'} above $\id$ denotes the $(3|\cT_\Delta|)\times(3|\cT_\Delta)$-identity matrix (with ones on the diagonal and zeros elsewhere).
Let $\|\vM(t)\|$ denote the spectral norm of $\vM(t)$.
The following proposition summarises the main combinatorial step towards the proof of \Thm~\ref{thm_main}.

\begin{proposition}\label{prop_condex}
	Let $\eps>0$ be independent of $n$.
	If 
	\begin{align}\label{eq_cond_prop_condex}
		\|\vM(t)\|&<1-\eps,&\min_{\chi,\theta}\vU_{\chi,\theta}(t)&\geq\eps n,&&\mbox{ and }&
		\min_{\chi,\theta}\vW_\theta(t)&\geq\eps n,
	\end{align}
 then for all $\theta=(s,d_1,d_2,d_3)\in\cT_\Delta$ we have
	\begin{align}\label{eq_prop_condex1}
		\ex\brk{\vW_\theta(t+1)-\vW_\theta(t)\mid\fD_t}=&- \sum_{s'=1}^3  \frac{d_{s'} \vec\kappa_{s',s}(t)\vW_\theta(t)}{\vE_{s,s'}(t)}+o(1),\\
		\ex\brk{\vU_{c,\theta}(t+1)-\vU_{c,\theta}(t)\mid\fD_t}
													   =& - \frac{(d_1+d_2+d_3)^\alpha \vU_{c,\theta}(t)}{\sum_{\theta'\in\cT_\Delta}(d_1'+d_2'+d_3')^\alpha \vU_{\theta'}(t)}\nonumber\\
													   &+
												   \sum_{s'=1}^3 \frac{\vec\kappa_{c, s',s}}{\vE_{s,s'}(t)}\brk{(d_{s'} + 1)(\vW_{\theta^{+s'}}(t) +\vU_{c,\theta^{+s'}} (t)) -d_{s'}\vU_{c,\theta}(t) }.\label{eq_prop_condex2}
	\end{align}
	Furthermore, the expected number of bad vertices is $O(1)$.
\end{proposition}
The proof of \Prop~\ref{prop_condex} can be found in Section\ref{sec_prop_condex}.

\subsection{Tracking $\algo$ via differential equations}\label{sec_diffeq}
We harness Wormald's method of differential equations~\cite{Wormald} to reduce the analysis of the combinatorial equations~\eqref{eq_prop_condex1}--\eqref{eq_prop_condex2} to a system of ODEs.
Specifically, let
$$
(w_\theta=w_\theta(t),u_{c,\theta}=u_{c,\theta}(t))_{c\in\{1,2,3\},\theta\in\cT_\Delta}
$$
be functions of $t\geq0$.
Define the following auxiliary functions $e_{s,s'}$, $\kappa_{c,s,s'}$, $\kappa_{s,s'}$, $M_{c',\theta',c,\theta}$ and $p_{c,\theta}$ for $c,c'\in\{1,2,3\}$ and $\theta=(s,d_1,d_2,d_3),\theta'=(s',d_1',d_2',d_3')\in\cT_\Delta$:
\begin{align}
	p_{c, \theta}  &= \frac{(d_1+d_2+d_3)^\alpha\sum_{\chi\in\{1,2,3\}\setminus\{c\}} u_{\chi,\theta}}{2\sum_{c'=1}^3\sum_{\theta'\in\cT_\Delta}(d_1'+d_2'+d_3')^\alpha u_{c',\theta'}}\label{eqpctheta},\\
	e_{s,s'}&=\sum_{d_1',d_2',d_3'=0}^{\Delta}\vecone\{\theta'\in\cT_\Delta\}d_s'\brk{w_{\theta'}+\sum_{c'=1}^3u_{c,\theta'}},\label{eqess'}\\
	M_{c', \theta', c, \theta} &= \vecone\{c\neq c'\}{d_{s'} (d_{s}'+1) u_{c'',{\theta'}^{+s}}}/{e_{s,s'}},\label{eqMmatrix}\\
	\kappa_{c,s,s'}&=\sum_{d_1,d_2,d_3=0}^\Delta \vecone\{\theta\in\cT_\Delta\} d_{s'}\left(\left(\id -M \right)^{-1} p\right)_{(c, \theta)},\label{eqkappacss'}\\
	\kappa_{s, s'}  &= \sum_{c=1}^3 \kappa_{c, s,s'},\label{eqkappass'}
\end{align}
where once again $M=(M_{c',\theta',c,\theta})_{(c',\theta'),(c,\theta)}$ is a $(3|\cT_\Delta|)\times(3|\cT_\Delta|)$-matrix and $p=(p_{c,\theta})_{c,\theta}$ is a row vector of size $3|\cT_\Delta|$.
Recalling the functions $u_\ell,w_\ell$ from~\eqref{eqODEinitial}--\eqref{eqODE}, consider the system of first-order ODEs
\begin{align}\label{eqODEcomplicated1}
	\frac{\dd w_\theta}{\dd t}&=-\sum_{s'=1}^3\frac{d_{s'}\kappa_{s,s'}w_\theta}{e_{s,s'}},\\
	\frac{\dd u_{c,\theta}}{\dd t}&=-\frac{(d_1+d_2+d_3)^\alpha u_{c,\theta}(t)}{\sum_{c'=1}^3\sum_{\theta'\in\cT_\Delta}(d_1+d_2+d_3)^\alpha u_{c',\theta'}}+\sum_{s'=1}^3\frac{\kappa_{c,s',s}}{e_{s,s'}}\brk{(d_{s'}+1)(w_{\theta^{+s'}}+u_{c,\theta^{+s'}})-d_{s'}u_{c,\theta}}\label{eqODEcomplicated2}
\end{align}
with initial conditions
\begin{align}\label{eqODEinitial_full}
	u_{c, \theta}(0) =& \frac 1 9 u_{d_1 + d_2 + d_3}(0)\binom{d_1 +d_2 + d_3}{d_1, d_2, d_3} \prod_{s'=1}^3 \bc{\frac{\exp(-\beta\vecone\cbc{s = s'})}{2+\exp(-\beta)}}^{d_{s'}},\\
	w_{\theta}(0) =& \frac 1 3 w_{d_1 + d_2 + d_3}(0)\binom{d_1 + d_2 + d_3}{d_1, d_2, d_3} \prod_{s'=1}^3 \bc{\frac{\exp(-\beta\vecone\cbc{s=s'})}{2+\exp(-\beta)}}^{d_{s'}}.\label{eqODEinitial_full2}
\end{align}

\begin{fact}\label{prop_diffeq}
	Assume that there exists $\eps>0$ such that the system of differential equations~\eqref{eqODEcomplicated1}--\eqref{eqODEcomplicated2} has a solution on an interval $[0,t^*+\eps]$ such that
	\begin{align}\label{eq_prop_diffeq}
		\inf_{t\in [0,t^*+\eps]}\min\{u_{c,\theta},w_\theta:c\in\{1,2,3\},\,\theta\in\cT_\Delta\}&>\eps&&\mbox{and}&
		\sup_{t\in [0,t^*+\eps]}\|M\|&<1-\eps.
	\end{align}
	Then \whp\ for all $t\leq t^*n$, all $\theta\in\cT_\Delta$ and all $c\in\{1,2,3\}$ we have
	\begin{align*}
		\vW_\theta(t)&=nw_\theta(t)+o(n),&\vU_{c,\theta}(t)&=n u_{c,\theta}(t)+o(n).
	\end{align*}
\end{fact}
\begin{proof}
	This is a standard application of the method of differential equations \cite[\Thm~2]{Wormald}.
	Specifically, Fact~\ref{prop_Delta} and \Prop~\ref{prop_high} yield the initial conditions~\eqref{eqODEinitial_full}--\eqref{eqODEinitial_full2}.
	Moreover, rescaling the conditional expectations from~\eqref{eq_prop_condex1}--\eqref{eq_prop_condex2} yields the equations~\eqref{eqODEcomplicated1}--\eqref{eqODEcomplicated2}.
\end{proof}

\subsection{The flat-white solution}\label{sec_flat}
As a next step we construct a solution to~\eqref{eqODEcomplicated1}--\eqref{eqODEinitial_full2} that satisfies \eqref{eq_prop_diffeq}, provided that Assumption~\ref{assumption_ode} is satisfied.
We refer to this solution as the {\em flat-white solution} to the ODEs.
Due to the uniqueness theorem for ODEs, the flat-white solution is in fact the {\em only} solution to \eqref{eqODEcomplicated1}--\eqref{eqODEinitial_full2}.
The construction of the flat-white solution is based on the conception that throughout the execution of $\algo$ the colour list assigned to a vertex $v$ is asymptotically independent of the planted colour $\SIGMA^*_v$.
More precisely, guided by this heuristic assumption we harness the functions $u_\ell(t),w_\ell(t)$ from Assumption~\ref{assumption_ode} to craft a solution to \eqref{eqODEcomplicated1}--\eqref{eqODEcomplicated2}.

\begin{proposition}\label{prop_flat}
	Under Assumption~\ref{assumption_ode} the following is true.
	With $\theta=(s,d_1,d_2,d_3)$ the functions 	
	\begin{align}\label{eq_prop_flat_w}
		w_{\theta}(t) = & \frac {w_{d_1+d_2+d_2}} 3 \binom{d_1+d_2+d_3}{d_1, d_2, d_3}  \prod_{s'=1}^{3} \bc{\frac{\eul^{-\beta \vecone\brk{s=s'}}}{2+\eul^{-\beta}}}^{d_{s'}},\\
		u_{c, \theta}(t) = & \frac {u_{d_1+d_2+d_2}} 9 \binom{d_1+d_2+d_3}{d_1, d_2, d_3}  \prod_{s'=1}^{3} \bc{\frac{\eul^{-\beta \vecone\brk{s=s'}}}{2+\eul^{-\beta}}}^{d_{s'}}\label{eq_prop_flat_u}
	\end{align}
	solve the system of differential equations \eqref{eqODEcomplicated1}--\eqref{eqODEcomplicated2} and satisfy~\eqref{eq_prop_diffeq}.
\end{proposition}

The delicate point that we need to verify toward the proof of \Prop~\ref{prop_flat} is that the conditions from Assumption~\ref{assumption_ode} do indeed suffice to guarantee the conditions~\eqref{eq_prop_diffeq}.
The proof of \Prop~\ref{prop_flat} can be found in \Sec~\ref{sec_prop_flat}.

\subsection{Cleaning up}\label{sec_cleanup}
Fact~\ref{prop_diffeq} and \Prop~\ref{prop_flat} track the $\algo$ algorithm for its first $t^*n$ epochs, with $t^*$ from Assumption~\ref{assumption_ode}.
However, as we discussed in \Sec~\ref{sec_deferred} by this time the algorithm will not yet have coloured all the vertices.
For example, $\algo$ does not touch any vertices that lie in components where all vertices have initial lists of size three.
In addition, by time $t^*n$ $\algo$ will likely not even be finished with the giant component of the initial graph $\G^*_\Delta$.
However, the following proposition shows that \whp\ the graph at time $t^*n$ likely decomposes into small components that can be coloured independently.

\begin{proposition}\label{prop_cleanup}
	Under Assumption~\ref{assumption_ode} the largest component of the graph $\G^*_{\Delta,t}$ with $t=\lfloor t^*n\rfloor $ has size $O(\log n)$ and all but $O(\log n)$ components of $\G^*_{\Delta,t}$ are acyclic.
	Furthermore, \whp\ for a random proper 3-colouring $\vec\tau$ of the acyclic components of $\G^*_{\Delta,t}$ such that every vertex $v$ receives a colour $\vec\tau(v)\in\cL_v(t)$ we have
	\begin{align}\label{eqprop_cleanup}
		\sum_{s,s'=1}^3\sum_v\bc{\vecone\{\vec\tau_v=s,\SIGMA^*(v)=s'\}-\frac19}&=o(n).
	\end{align}
\end{proposition}

The proof of \Prop~\ref{prop_cleanup}, which can be found in \Sec~\ref{sec_prop_cleanup}, hinges on condition~\eqref{eqODEend}, which resembles the well-known Molloy--Reed condition for subcriticality of random graphs with given degrees~\cite{MolloyReed}.

Finally, we observe that the colouring obtained via the strategy outlined in this section has vanishing agreement with $\SIGMA^*$ \whp

\begin{corollary}\label{cor_overlap}
	Under Assumption~\ref{assumption_ode} the following is true.
	Obtain $\SIGMA^\dagger\to\{1,2,3\}$ as follows.
	\begin{itemize}
		\item For all $v\in V^\#\setminus V^\#_\Delta$, let $\SIGMA^\dagger_v=\sigma^\#_v$, with $\sigma^\#$ from \Prop~\ref{prop_high}.
		\item Let $\vec\tau$ agree with the partial colouring produced by running $\algo$ on $\G_\Delta^*$.
		\item For all remaining vertices let $\SIGMA^\dagger_v=\vec\tau_v$, with $\vec\tau$ from \Prop~\ref{prop_cleanup}.
	\end{itemize}
	Then \whp\ we have
	\begin{align}\label{eq_cor_overlap}
		\agree(\SIGMA^\dagger)&=o(1)&\mbox{and}&&
		\sum_{v\in V_n}\vecone\{\SIGMA^{\dagger}(v)=i\}&\sim\frac n3\mbox{ for all }i\in\{1,2,3\}.
	\end{align}
\end{corollary}
\begin{proof}
	The assertion is an immediate consequence of \Prop~\ref{prop_high} (iii), Equation~\eqref{eqprop_cleanup} and the fact that the functions $w_\theta(t),u_{c,\theta}(t)$ from~\eqref{eq_prop_flat_w}--\eqref{eq_prop_flat_u} and therefore the colours assigned by $\algo$ are independent of the planted colour~$s$.
\end{proof}

\noindent
At this point the proofs of \Thm~\ref{thm_main} and \Cor~\ref{cor_main} are straightforward from the above propositions.
The details can be found in \Sec~\ref{sec_proof_main}.

\section{Proof of \Prop~\ref{prop_high}}\label{sec_prop_high}

\noindent
The starting point of the proof of \Prop~\ref{prop_high} is the well known fact that for large enough $\Delta$ the graph $\G^\#$ is sub-critical, as summarised by the following lemma.

\begin{lemma}\label{lem_high_subcrit}
	For any $d>0$ there exists $\Delta_0>0$ such that for all $\Delta\geq\Delta_0$ and all $\beta\geq0$,
	\begin{itemize}
		\item the largest component of the graph $\G^\#$ contains $O(\log n)$ vertices, and
		\item the expected number of cyclic components of $\G^\#$ is $O(1)$.
	\end{itemize}
	Furthermore, with probability $\Omega(1)$ the random graph $\G^\#$ is acyclic.
\end{lemma}
\begin{proof}
	This follows from the routine observation that the local topology of $\G^\#$ follows a two-type branching process.
	The two types are vertices of degree greater than $\Delta$, called {\em big} vertices, and vertices of degree at most $\Delta$, called {\em small} vertices.
	To work out the offspring distributions, let $\vec B,\vec S$ be random variables with size-biased conditional Poisson distributions
	\begin{align*}
		\pr\brk{\vec B=h}&=\frac{h\pr\brk{\Po_{>\Delta}(d)=h}}{\ex[\Po_{>\Delta}(h)]},&
		\pr\brk{\vec S=h}&=\frac{h\pr\brk{\Po_{\leq\Delta}(d)=h}}{\ex[\Po_{\leq\Delta}(h)]}.
	\end{align*}
	Then the matrix $\cM=\begin{pmatrix}
		\cM_{\mathrm{big},\mathrm{big}}&\cM_{\mathrm{big},\mathrm{small}}\\
		\cM_{\mathrm{small},\mathrm{big}}&\cM_{\mathrm{small},\mathrm{big}}
		\end{pmatrix}$ 
	with the following entries captures the expected offsprings:
	\begin{align*}
		\cM_{\mathrm{big},\mathrm{big}}&=\frac{\ex\brk{\vec B-1}\ex\brk{\Po_{>\Delta}(d)}}d,&
		\cM_{\mathrm{small},\mathrm{big}}&=\frac{\ex\brk{\vec B-1}\ex\brk{\Po_{\leq\Delta}(d)}}d,\\
		\cM_{\mathrm{small},\mathrm{big}}&=\frac{\ex\brk{\vec S-1}\ex\brk{\Po_{>\Delta}(d)}}d,&
		\cM_{\mathrm{small},\mathrm{small}}&=0.
	\end{align*}
The initial distribution upon exploring from a random edge emerging from high-degree vertex simply reads
\begin{align}
	q_{\mathrm{big}} &= 1,&
	q_{\mathrm{small}} &= 0.
\end{align}
	Since the Poisson distribution has subexponential tails, we obtain $\|\cM\|<1$ for sufficiently large $\Delta$, which implies sub-criticality.
	As a consequence, the expected number of cyclic components is bounded, and with probability $\Omega(1)$ all components are acyclic.
\end{proof}

\begin{proof}[Proof of \Prop~\ref{prop_high}]
	In light of \Lem~\ref{lem_high_subcrit}, we proceed as follows to obtain a colouring $\SIGMA^\#$ of $\G^\#$.
	Pick any uncoloured vertex $r$ of degree greater than $\Delta$ (so long as one exists) and choose a random colour $\SIGMA^\#(v)\in\{1,2,3\}$; due to \Lem~\ref{lem_high_subcrit} we may assume that the component $\cC(r)$ of $r$ in $\G^\#$ is a tree.
	Let $\cV(r)$ be the set of vertices contained in the component of $r$.
	Further, extend $\SIGMA^\#$ to a colouring of $\cV(r)$ by choosing a random colouring of $\cC(r)$ with (at most) two colours from $\{1,2,3\}$.
	Then \whp\	the resulting colouring $\SIGMA^\#$ has the properties (i)--(iii).
	Finally, if $\G^\#$ is acyclic, then $\SIGMA^\#$ is a proper colouring of $\G^\#$.
\end{proof}

\section{Proof of \Prop~\ref{prop_condex}}\label{sec_prop_condex}

\noindent
In epoch $t$ of $\algo$ we execute Step~3 once and then subsequently Steps~5--8 a number of times.
During the single execution of Step 3 one random vertex $\vf_t$ with $\abs{\cL_{\vf_t}(t)}=2$ receives a random colour $\sigma(u)\in\cL_{\vf_t}(t)$.
We refer to the subsequent executions of Steps 4--8 as the {\em forced colourings} of epoch $t$.
Every time a vertex gets coloured, its colour gets removed from the lists of all its neighbours.
This may produce new vertices with only one colour left in their lists, etc.  

Let $\fF_t$ be the set of all vertices (including $\vf_t$) that receive a colour during epoch $t$.
By construction, for each vertex $u\in\fF_t\setminus\{\vf_t\}$ there exists another vertex $\fp(u)\in\fF_t$ such that colouring $\fp(u)$ with colour $\sigma({\fp(u)})$ reduced the list of colours available to $u$ to size one.
We call $\fp(u)$ the {\em parent vertex} of $u$.
Let $\fT_t$ be the directed graph rooted at $\vf_t$ that contains the edge $(\fp(u),u)$ for every $u\in\fF_t \setminus \cbc{\vf_t}$.
Furthermore, let
\begin{align*}
	\vK_{c,s,s'}(t)&=\sum_{v\in\fF_t}\vecone\{\SIGMA^*(v)=s,\,\sigma(v)=c\}\brk{\vd_{v,s'}-\vecone\{v\neq\vf_t,\,\SIGMA^*(\fp(v))=s'\}}
\end{align*}
be the number of edges that connect a vertex $v\in\fT_t$ with $\sigma(v)=c$ and $\SIGMA^*(v)=s$ with a vertex $w\not\in\fT_t$ with planted colour $\SIGMA^*(w)=s'$.

\begin{lemma}\label{lem_kappa}
	If~\eqref{eq_cond_prop_condex} is satisfied, then $\ex\brk{\vec K_{c,s,s'}\mid\fD(t)}=\vec\kappa_{c,s,s'}(t)+o(1)$.
\end{lemma}
\begin{proof}
	Extending the argument used in~\cite{AchMoore3col} for \Erdos-\Renyi\ random graphs to the \sbm, we construct a multi-type branching process that mimics the directed graph $\fT_t$.
Due to~\eqref{eq_cond_prop_condex} this branching process will turn out to be sub-critical.
As a consequence, \whp\ $\fT_t$ is a tree of bounded size, the subgraph of $\G_\Delta^*(t)$ induced on $\fF_t$ is sub-critical and the conditional mean of $\vK_{c,s,s'}(t)$ can be calculated from the total progeny of the branching process.

The parameters of the branching process depend on $t, \vU_{c, \theta}, \vW_{\theta}$ only, which are $\fD(t)$-measurable.
The types of the branching process are pairs $(c,\theta)$ with $c\in\{1,2,3\}$ and $\theta\in\cT_\Delta$.
The intended semantics is that $c$ indicates the colour $\sigma(u)$ that $\algo$ assigned to the corresponding vertex $u$ of $\fT_t$.
Moreover, $\theta=(s,d_1,d_2,d_3)$ comprises the planted colour $\SIGMA^*(u)=s$ of $u$, and $d_i$ equals the number of neighbours $v\neq\fp(u)$ of $u$ in $\G_\Delta^*(t)$ with planted colour $\SIGMA^*(v)=i$.

The root of the branching process is the vertex $\vf_t$ with two available colours $\cL_{\vf_t}(t)=\{c,c'\}$.
Recalling the choice of $\vf_t$ in Step~3 of $\algo$ and~\eqref{eqpctheta}, we see that the probability that $\vf_t$ has type $(c,\theta)$ equals
\begin{align}\nonumber
	\Pr&\brk{\sigma(\vf_t) = c, \SIGMA^*(\vf_t)=s, \vd_{\vf_t,1}(t) = d_1, \vd_{\vf_t,2} (t)= d_2, \vd_{\vf_t,3} (t)= d_3 \mid \fD_t}\\ &= \frac{(d_1+d_2+d_3)^\alpha \sum_{\chi \in \set{1,2,3}\setminus\set{c}}\vU_{\chi,\theta}(t)}{\sum_{c'} \sum_{\theta' \in \cT_\Delta } \sum_{c''\in \set{1,2,3}\setminus \set{c'}}(d_1'+d_2'+d_3')^\alpha \vU_{c'',\theta'}(t)}
	=  \frac 1 2 \frac{(d_1+d_2+d_3)^\alpha\sum_{\chi\in\{1,2,3\}\setminus\{c\}} \vU_{\chi,\theta}(t)}{\sum_{\theta'\in\cT_{\Delta}}(d_1'+d_2'+d_3')^\alpha\vU_{\theta'}(t)}=\vp_{c,\theta}(t).\label{eq_lem_kappa_1}
\end{align}

%

As a next step we determine the offspring distribution of the branching process.
Specifically, given that the root vertex $\vf_t$ has type $(c,\theta)$, we are going to compute the mean $\mu_{c',\theta',c,\theta}$ of the number of $\vX_{c',\theta'}$ of children of $\vf_t$ of type $(c',\theta')$.
Thus, $\mu_{c',\theta',c,\theta}$ is the expected number of vertices $v\in\partial\vf_t$ of type $\theta'$ that had two colours in their list $\cL_v(t)$ before $\vf_t$ got assigned colour $c$, but that pursuant to this decision have only a single colour left in their list.
Recall that the type $\theta' = (s', d_1', d_2', d_3')$ of the children discounts the edge that joins the child to the parent.
In particular, the total number of neighbours of $v$ with planted colour $s$ in $\G_\Delta^*(t)$ equals $d_s'+1$.
Therefore, invoking Fact~\ref{prop_deferred} (the principle of deferred decisions), we claim that
\begin{align}\label{eq_lem_kappa_2}
	\mu_{c',\theta',c,\theta}&\sim\frac{\vecone\{c\neq c'\}{d_{s'} (d_{s}'+1) \vU_{c'',{\theta'}^{+s}}(t)}}{\sum_{v\in V(t):\SIGMA^*(v)=s'}\vd_{v,s}(t)}&&(c''\in\{1,2,3\}\setminus\{c,c'\}).
\end{align}
Indeed, since $\G_\Delta^*(t)$ is uniformly random given $\fD(t)$, given that $\vf_t$ has $s'$ neighbours with planted colour $s'$ the total expected number of possible edges between $\vf_t$ and vertices of type $(c',{\theta'}^{+s})$ equals $d_{s'} (d_{s}'+1) \vU_{c'',{\theta'}^{+s}}(t)$.
Moreover, to obtain the expected number of neighbours we need to normalise by the total number of edges $xy$ with $\SIGMA^*_x=s$ and $\SIGMA^*_y=s'$, which explains the denominator in~\eqref{eq_lem_kappa_2}.

We notice that the asymptotic formula~\eqref{eq_lem_kappa_2} yields the expected offspring not only of the root $\vf_t$, but also of any vertex $u$ that gets coloured during the first $o(n)$ forced steps of epoch $t$.
For while strictly speaking we should discount the previously revealed edges when computing the expected offspring of $u$, the total number of such edges is only of order $o(n)$.
Hence, under assumption \eqref{eq_cond_prop_condex} the expectation remains asymptotically the same.
Furthermore, recalling~\eqref{eqEss'}--\eqref{eqMmatrix}, we see that $\mu_{c',\theta',c,\theta}=\vM_{c',\theta',c,\theta}(t)$.

In summary, the first $o(n)$ forced colourings during epoch $t$ can be coupled with a branching process with root type distribution $\vp(t)$ and offspring matrix $\vM(t)$.
Since \eqref{eq_cond_prop_condex} ensures that $\vM(t)$ has spectral radius strictly less than one \whp, the branching process is sub-critical \whp.
Therefore, so is the process of pursuing forced colourings in epoch $t$.
In particular, the sub-graph $\fT_t$ has bounded expected size, and is therefore acyclic \whp\
Finally, the standard formula for the total progeny of a branching process yields
\begin{align*}
	\ex[\vK_{c,s,s'}\mid\fD(t)]&=\sum_{d_1,d_2,d_3=0}^{\Delta} d_{s'}\vecone\{(s,d_1,d_2,d_3)\in\cT_\Delta\} \sum_{\ell=0}^\infty (\vM^\ell(t) \vp(t))_{(c, (s, d_1, d_2, d_3))}\\
							   &=\sum_{d_1,d_2,d_3=0}^\Delta d_{s'}\vecone\{(s,d_1,d_2,d_3)\in\cT_\Delta\}\left(\left(\id -\vM(t) \right)^{-1}\vec p(t)\right)_{(c, (s, d_1 , d_2, d_3))}=\vec\kappa_{c, s, s'}(t), 
\end{align*}
as claimed.
\end{proof}

\begin{proof}[Proof of \Prop~\ref{prop_condex}]
	As a first step we verify the formula~\eqref{eq_prop_condex1} for the expected change of the number of `white' vertices of type $\theta=(s,d_1,d_2,d_3)$.
	Clearly, $\vW_{\theta}(\nix)$ is monotonically decreasing in $t$: the only possible changes ensue because some `white' vertices $w$ of type $\theta$ lose one of their available colours because a neighbour $v\in\partial w$ gets assigned a colour $\sigma(w)$ in epoch $t$ (either because $w=\vf_t$ or because of a forced assignment).
Hence, we just need to calculate the expected total number of white vertices of type $\theta$ that have a neighbour in $\fT_t$.
This is easy.
Indeed, the total number of edges between vertices with planted colour $s'$ in $\fT_t$ and vertices $w$ with planted colour $s$ equals $\sum_{c=1}^3\vK_{c,s,s'}$.
Furthermore, if $w$ has type $\theta$, then the number of edges between $w$ and the set of vertices with planted colour $s'$ equals $d_{s'}$. 
Hence, Fact~\ref{prop_deferred}, \Lem~\ref{lem_kappa} and~\eqref{eqkappacss'}--\eqref{eqkappass'} yield
\begin{align}\label{eq_prop_condex_0}
	\ex\brk{\vW_\theta(t+1)-\vW_\theta(t)\mid\fD_t}&
	=-\sum_{s'=1}^3\sum_{c=1}^3d_{s'}\vW_\theta(t)\ex\brk{\vK_{c,s,s'}(t)\mid\fD_t}
	=- \sum_{s'=1}^3  \frac{d_{s'} \vec\kappa_{s',s}(t)\vW_\theta(t)}{\vE_{s,s'}(t)}+o(1).
\end{align}

Next we calculate the expected change of $\vU_{c, \theta}(\nix)$ for some colour $c\in\{1,2,3\}$ and type $\theta = (s, d_1, d_2, d_3)\in\cT_\Delta$.
When epoch $t$ commenced these vertices had allowed colours $\set{1,2,3}\setminus\set{c}$.
The random variable $\vU_{c,\theta}(\nix)$ can decrease for one of two reasons.
First, it could be that $\vf_t$ counted towards $\vU_{c,\theta}(t)$.
Step~3 of $\algo$ ensures that the probability of this event equals
\begin{align}\label{eq_prop_condex_1}
	- \frac{(d_1+d_2+d_3)^\alpha \vU_{c,\theta}(t)}{\sum_{\theta'\in\cT_\Delta}(d_1'+d_2'+d_3')^\alpha \vU_{\theta'}(t)}.
\end{align}
Second, it could be that a vertex $v$ that had type $\theta$ and list $\cL_v(t)=\{1,2,3\}\setminus\{c\}$ at the beginning of epoch $t$ is incident with an edge emanating from the tree $\fT_t$.
Similar reasoning as towards~\eqref{eq_prop_condex_0} shows that this occurs with probability
\begin{align}\label{eq_prop_condex_2}
	-\sum_{s'=1}^3 \frac{\vec\kappa_{c, s',s} d_{s'}\vU_{c,\theta}(t) }{\vE_{s,s'}(t)}.
\end{align}

Finally, by contrast to $\vW_\theta(\nix)$, the random variables $\vU_{c,\theta}(\nix)$ might also increase.
This may happen in one of two ways.
First, a vertex $w$ that was white at the beginning of epoch $t$ and that is incident to an edge emerging from $\fT_t$ may turn into a vertex of type $\theta$ with list $\cL_v(t+1)=\{1,2,3\}\setminus\{c\}$.
In that case $w$ was of type $\theta^{+s'}$ at the beginning of epoch $t$ and is adjacent to a vertex $v$ of $\fT_t$ of planted colour $s'$.
Fact~\ref{prop_deferred} and \Lem~\ref{lem_kappa} show that the expected ensuing increase comes to
\begin{align}\label{eq_prop_condex_3}
	\sum_{s'=1}^3 \frac{\vec\kappa_{c, s',s}(d_{s'} + 1)\vW_{\theta^{+s'}}(t)  }{\vE_{s,s'}(t)}+o(1).
\end{align}
A further possibility is that a vertex $u$ with $|\cL_{u}|=2$ and type $\theta^{+s'}$ is incident with an edge emanating from a vertex $v$ of planted colour $s'$ of $\fT_t$.
In analogy to~\eqref{eq_prop_condex_3} the expected ensuing increase reads
\begin{align}\label{eq_prop_condex_4}
	\sum_{s'=1}^3 \frac{\vec\kappa_{c, s',s}(d_{s'} + 1)\vU_{\theta^{+s'}}(t)  }{\vE_{s,s'}(t)}+o(1).
\end{align}
Combining~\eqref{eq_prop_condex_1}--\eqref{eq_prop_condex_4}, we obtain~\eqref{eq_prop_condex2}.

Finally, we need to estimate the number of bad vertices.
Every bad vertex that emerges in an epoch has to have had at least two colours removed from its list. 
Thus, every bad vertex lies on a cycle of $\fT_t$. 
Since \eqref{eq_cond_prop_condex} ensures that $\|\vM(t)\|<1-\eps$ while at least a linear number $\Omega(n)$ of uncoloured vertices remain, in every epoch $t$ probability that $\fT_t$ contains a cycle is of order $O(1/n)$.
Consequently, the expected number of bad vertices is $O(1)$.
\end{proof}

\section{Proof of \Prop~\ref{prop_flat}}\label{sec_prop_flat}

\noindent
We begin by verifying that \eqref{eq_prop_diffeq} is satisfied.
To this end, for $\theta=(s,d_1,d_2,d_3)\in\cT_\Delta$ and $c',s_1',s_2'\in\{1,2,3\}$ define
\begin{align*}
	M^V_{(c',s_1',s_2'),\theta}=&\vecone\cbc{c=c',\,s=s_1'} d_{s_2'},&
	M^E_{\theta,(c',s_1',s_2')}=& \vecone\cbc{c\neq c',\,s=s_2'} \frac{(d_{s_1'}+1)u_{c,\theta^{+{s_1'}}}}{e_{s_1',s_2'}}.
\end{align*}
Then we obtain $27\times|\cT_\Delta|$- and $|\cT_\Delta|\times27$-matrices
\begin{align}\label{eqMEMV}
	M^V&=(M^V_{(c',s_1',s_2'),\theta})_{c',s_1',s_2'\in\{1,2,3\};\,\theta\in\cT_\Delta},&
	M^E&=(M^E_{\theta,(c',s_1',s_2')})_{\theta\in\cT_\Delta;\,c',s_1',s_2'\in\{1,2,3\}}.
\end{align}
Recalling the matrix from~\eqref{eqMmatrix} and performing a trite computation, we discover that
\begin{align}\label{eqMdecomp}
	M=M(t) = M^E M^V.
\end{align}
Finally, for $\ell\geq0$ we introduce the shorthands
\begin{align*}
	u^{(\ell)}&= \sum_{i=0}^{\Delta} i^\ell u_i,& w^{(\ell)}&= \sum_{i=0}^{\Delta} i^\ell w_i.
\end{align*}

\begin{lemma}\label{lem_eigenvector}
	Under Assumption~\ref{assumption_ode}	we have $\|M\|<1-\xi$ for all $t\in[0,t^*+\xi]$.
\end{lemma}
\begin{proof}
	The matrix $M$ is non-negative and irreducible.
	The Perron-Frobenius theorem therefore shows that $M$ has a single dominant eigenvalue $\lambda_1(M)>0$ with $\|M\|=\lambda_1(M)$.
	Thus, we need to show that $\lambda_1(M)<1-\xi$.
	To this end, we observe that $\|M^\ell\|=\lambda_1(M)^\ell$ for every $\ell\geq1$.
	Hence, to show that $\lambda_1(M)<1-\xi$ for a small $\xi=\Omega(1)$ it suffices to prove that $\|M^\ell\|<1-\Omega(1)$ for large enough $\ell$.
	In order to bound $\|M^\ell\|$ we use the decomposition~\eqref{eqMdecomp}.
	Specifically, using \eqref{eqMdecomp} and the sub-multiplicativity of the matrix norm, we obtain
	\begin{align*}
		\|M^\ell\|=\|(M^EM^V)^\ell\|\leq\|M^E\|\cdot\|M^VM^E\|^{\ell-1}.
	\end{align*}
	Hence, taking $\ell$ large, we see that it suffices to prove that $\|M^VM^E\|<1-\Omega(1)$.

	To this end, we compute the entries of $M^VM^E$.
	Recalling $e_{s,s'}$ from~\eqref{eqess'}, we first compute
	\begin{align}\label{eq_prop_flat_1}
		e_{s,s'}=&\sum_{d'_1, d'_2, d'_3}\vecone\{\theta' \in\cT_\Delta\}d'_s\bc{w_{(s', d'_1, d'_2, d'_3)}+\sum_{c'=1}^3u_{c,(s',d'_1, d'_2, d'_3)}} =\frac 1 3  \frac{\eul^{-\beta \vecone\brk{s=s'}}}{2+\eul^{-\beta}} \sum_{\ell=1}^\Delta \bc{ w_{\ell}  + u_{\ell}} \ell . 
	\end{align}
	Hence, for $(c,s_1,s_2),(c',s_1',s_2')\in\{1,2,3\}^3$ we obtain 
	\begin{align}\nonumber
		(M^V M^E)_{(c',s_1', s_2')(c, s_1, s_2)} = 
		&\vecone\cbc{c\neq c'} \vecone\cbc{s_2 = s_1'}  \sum_{d_1,d_2,d_3=0}^{\Delta} \vecone\{\theta=(s_2,d_1,d_2,d_3)\in\cT_\Delta\} \frac{(d_{s_1}+1)d_{s_2'}}{e_{s_1,s_2}} u_{c,\theta^{+s_1}}\\
		=& \vecone\cbc{c\neq c'} \vecone\cbc{s_2 = s_1'} \frac{\mom u 2 -\mom u 1}{3(\mom u 1 + \mom w 1)} \frac{\eul^{-\beta\vecone\cbc{s_1'=s_2'}}}{2+\eul^{-\beta}} \label{eq:MVME}.
	\end{align}
	Due to the product structure of the expression~\eqref{eq:MVME} we can write $M^VM^E$ as a tensor product.
	Thus, writing $\mathbf 1$ for the $3\times 3$-matrix with all entries equal to one, we obtain
	\begin{align}
		\|M^VM^E\| 
		= &\abs{ \frac{u^{(2)} - u^{(1)} }{3(u^{(1)} - w^{(1)}) (2+\eul^{-\beta}) }}\cdot\norm{ \left(\mathbf1 - \id\right) \tensor \id \tensor 
		\bc{\begin{matrix}
				\eul^{-\beta} &  \eul^{-\beta} &  \eul^{-\beta} \\
				1 & 1 & 1 \\
				1 & 1 & 1 \\
\end{matrix}}}=\frac{2(u^{(2)} - u^{(1)}) }{3(u^{(1)} + w^{(1)})}.
\label{eqMVMEeig}
	\end{align}
	Finally, Assumption~\ref{assumption_ode} ensures that the r.h.s.\ of \eqref{eqMVMEeig} is bounded away from one.
\end{proof}

\begin{corollary}\label{cor_eigenvector}
	If Assumption~\ref{assumption_ode} holds, then \eqref{eq_prop_diffeq} is satisfied.
\end{corollary}
\begin{proof}
	The first condition is a direct consequence of the definitions~\eqref{eq_prop_flat_w}--\eqref{eq_prop_flat_u} of the functions $w,u$ and \eqref{eqODEend}.
	The second condition follows from \Lem~\ref{lem_eigenvector}.
\end{proof}

As a next step we are going to compute the quantities $\kappa_{c,s,s'}$ and $\kappa_{s,s'}$ from~\eqref{eqkappacss'}--\eqref{eqkappass'}.
Let
\begin{align}\label{eqkappalambda}
		\kappa&=\frac {3 (\mom u 1 + \mom w 1 ) \mom u {1+\alpha} }{(5 \mom u 1 + 3 \mom w 1 - 2 \mom u 2)\mom u \alpha}>0,&
		0<\lambda = & \frac{2(u^{(2)} - u^{(1)}) }{3(u^{(1)} + w^{(1)})}<1.
	\end{align}

\begin{lemma}\label{lemma_kappa}
	If Assumption~\ref{assumption_ode} holds, then for all $c,s,s'\in\{1,2,3\}$ we have
 \begin{align*}
	 \kappa_{c,s,s'}&= \frac\kappa9\cdot\frac{\exp\bc{-\beta \vecone\cbc{s = s'}}} {2+\exp\bc{-\beta}},&
	 \kappa_{s,s'}&= \frac\kappa3\cdot\frac{\exp\bc{-\beta \vecone\cbc{s = s'}}} {2+\exp\bc{-\beta}}.
 \end{align*}
\end{lemma}
\begin{proof}
	With $p$ from~\eqref{eqpctheta} and $M^V$ from \eqref{eqMEMV} we let $\hat p=M^Vp$.
	Recalling $M^V M^E$ from \eqref{eq:MVME}, we obtain
	\begin{align}\nonumber
		\hat{p}_{c',s_1, s_2} = & (M^V p)_{c',s_1,s_2} = \sum_{\substack{c\in\{1,2,3\}\\\theta=(s,d_1,d_2,d_3)\in\cT_\Delta}} \frac{\mom u {1+\alpha } }{9 \mom u \alpha }\frac{\eul^{-\beta \vecone\cbc{s_1 = s_2}}}{2+\eul^{-\beta}} \vecone\cbc{c \neq c'}\vecone \cbc{s_2 = s} \frac{(d_{s_1} +1) u_{c'\theta^{+s_1}}}{e_{s_1,s_2}} \\
		=& \frac{\mom u {1+\alpha}}{9 \mom u \alpha} \frac{\eul^{-\beta\vecone\cbc{s_1 = s_2}}}{2+\eul^{-\beta}}\qquad\mbox{[using~\eqref{eq_prop_flat_1}]}.\label{eqlemma_kappa_1}
	\end{align}
	Thanks to~\eqref{eqlemma_kappa_1}, we can now verify easily that $\hat p$ is an eigenvector of $M^V M^E$ with the eigenvalue $\lambda$:
	\begin{align}\nonumber
		\left(M^VM^E\hat p\right)_{(c_2, s_2, s_3)} 
		= & \sum_{c_1, s_1=1}^3\frac{\eul^{-\beta \vecone\cbc{s_1 = s_2}}}{9 u^{(\alpha)} (2+\eul^{-\beta})} \vecone\cbc{c_1 \neq c_2} \frac{u^{(2)} - u^{(1)}}{3(u^{(1)} + w^{(1)})} \frac{\eul^{-\beta \vecone \cbc{s_2 = s_3}}}{2+\eul^{-\beta}}\\
		= & \frac{2 (\mom u 2 - \mom u 1) \mom u {1+\alpha} \eul^{-\beta \vecone \cbc{s_2 = s_3}}}{3 \cdot 9 \mom u \alpha (\mom u 1 + \mom w 1 )(2+\eul^{-\beta})} =  \frac{2(\mom u 2 - \mom u 1)}{3 ( \mom u 1 + \mom w 1)} \hat p_{c_2, s_2, s_3}=\lambda\hat p_{c_2, s_2, s_3}.\label{eqlemma_kappa_2}
	\end{align}
	Combining~\eqref{eqlemma_kappa_2} with the identity $(\id-M)^{-1}=\sum_{\ell\geq0}M^\ell$ and \eqref{eqMdecomp}, we obtain
	\begin{align}\label{eqlemma_kappa_3}
		\kappa_{c, s,s'} =& \sum_{d_1,d_2,d_3=0}^{\Delta} \vecone\{(s,d_1,d_2,d_3)\in\cT_\Delta\}  d_{s'} \bc{p + \frac{1}{1-\lambda} M^E \hat{p} }_{c,s, d_1,d_2,d_3}.
	\end{align} 
Before we proceed to calculate $\kappa$, we notice that its constituent parts $(1-\lambda)^{-1}$ and $M^E \hat p$ boil down to
\begin{align*}
	\frac{1}{1-\lambda} &=  \frac{3 (\mom u 1 + \mom w 1)}{ 5 \mom u 1 + 3 \mom w 1 - 2 \mom u 2},\\
	M^E \hat{p}_{c', \theta'} &=  \frac{2}{27} \frac{\mom u {1+\alpha}}{\mom u \alpha}\frac{u_{d_1'+d_2'+d_3'+1} (d_1'+d_2'+d_3''+1)}{\mom u 1 + \mom w 1} \binom{d_1'+d_2'+d_3'}{d_1',d_2',d_3'} \prod_{s''} \bc{\frac{\eul^{-\beta\vecone\brk{s''=s'}}}{2+\eul^{-\beta}}}^{d'_{s''}}.
\end{align*}
Hence,
\begin{align*}
	\kappa_{c, s,s'} = & \sum_{d_1,d_2,d_3=0}^{\Delta} \vecone\{\theta \in\cT_\Delta\} d_{s'} \bc{p + \frac{1}{1-\lambda} M^E \hat{p} }_{c,s, d_1,d_2,d_3}\\
	=& \sum_{d_1,d_2,d_3=0}^{\Delta} \vecone\{\theta \in\cT_\Delta\} d_{s'} p_{c,s, d_1,d_2,d_3} + \sum_{d_1,d_2,d_3=0}^{\Delta} d_{s'} \bc{ \frac{1}{1-\lambda} M^E \hat{p} }_{c,s, d_1,d_2,d_3}\\
	=& \sum_{d_1,d_2,d_3=0}^{\Delta} \vecone\{\theta \in\cT_\Delta\} d_{s'} (d_1+d_2+d_3)^\alpha \frac{u_{d_1+d_2+d_3} }{9 \mom u \alpha} \binom{d_1+d_2+d_3}{d_1,d_2,d_3} \prod_{s''=1}^3 \bc{\frac{\eul^{-\beta\vecone\brk{s=s''}}}{2+\eul^{-\beta}}}^{d_{s''}} \\
	&+  \sum_{d_1,d_2,d_3=0}^{\Delta} \vecone\{\theta \in\cT_\Delta\} d_{s'} \bigg( \frac{3 (\mom u 1 + \mom w 1)}{ 5 \mom u 1 + 3 \mom w 1 - 2 \mom u 2} \frac{2}{27} \frac{\mom u {1+\alpha}}{\mom u \alpha}\frac{u_{d_1+d_2+d_3 + 1} (d_1+d_2+d_3+1)}{\mom u 1 + \mom w 1} \\& \qquad\qquad\binom{d_1+d_2+d_3}{d_1,d_2,d_3} \prod_{s''} \bc{\frac{\eul^{-\beta\vecone\brk{s''=s'}}}{2+\eul^{-\beta}}}^{d_{s''}}\bigg)\\
	=& \frac 1 9 \frac{\eul^{-\beta \vecone\brk{s=s'}}}{2+\eul^{-\beta}}\frac{\mom u {1+\alpha}}{\mom u \alpha}  +  \frac 2 9 \frac{ \mom u {1+\alpha} }{\mom u \alpha } \frac{\eul^{-\beta \vecone\brk{s=s'}}}{2+\eul^{-\beta}} \frac {\mom u 2 - \mom u 1 }{5 \mom u 1 + 3 \mom w 1 - 2 \mom u 2 }
	=  \frac 1 9 \frac{\eul^{-\beta \vecone\brk{s=s'}}}{2+\eul^{-\beta}} \kappa.
\end{align*}
Finally, we obtain
\begin{align}
	\kappa_{s, s'} = & \sum_{c = 1}^3 \kappa_{c, s, s'} =  \frac 1 3 \frac{\eul^{-\beta \vecone\brk{s=s'}}}{2+\eul^{-\beta}} \kappa,
\end{align}
thereby completing the proof.
\end{proof}

\begin{proof}[Proof of \Prop~\ref{prop_flat}]
	In light of \Cor~\ref{cor_eigenvector}, the remaining task is to verify \eqref{eqODEcomplicated1}--\eqref{eqODEcomplicated2}.
	To this end, we simply differentiate $w_\theta$ and $u_{c,\theta}$ from \eqref{eq_prop_flat_w}--\eqref{eq_prop_flat_u}.
	Hence, recalling $\kappa$ from~\eqref{eqkappalambda} and using \Lem~\ref{lemma_kappa} and \eqref{eq_prop_flat_1}, we obtain 
\begin{align*}
	\frac{\dd w_{s, d_1, d_2, d_3}}{\dd t} = & -\frac {\kappa \bc{{d_1+d_2+d_3}} w_{d_1+d_2+d_3}}  {3\sum_{\ell=1}^\Delta \ell(u_\ell+w_\ell)} \binom{d_1+d_2+d_3}{d_1, d_2, d_3}  \prod_{s'=1}^{3} \bc{\frac{\eul^{-\beta \vecone\brk{s=s'}}}{2+\eul^{-\beta}}}^{d_{s'}}\\
	= & - \frac{ \kappa \bc{{d_1+d_2+d_3}} w_{s, d_1, d_2, d_3} }{\sum_{\ell=1}^\Delta \ell(u_\ell+w_\ell)} = - \sum_{s'=1}^3\frac{   d_{s'} \kappa_{s',s} w_{s, d_1, d_2, d_3} }{e_{s,s'}},
\end{align*}
thereby verifying \eqref{eqODEcomplicated1}.

Moving on to \eqref{eqODEcomplicated2} and setting $\theta=(s,d_1,d_2,d_3)$, we obtain
\begin{align*}
	\frac{\dd u_{c,\theta}}{\dd t}=& 
	 \frac19\Bigg(\frac{\kappa\vecone\{d_1+d_2+d_3<\Delta\}\bc{(d_1+d_2+d_3+1)(w_{d_1+d_2+d_3+1}+u_{d_1+d_2+d_3+1}/3)-\bc{d_1+d_2+d_3} u_{d_1+d_2+d_3}}}{\sum_{i=1}^\Delta i(u_i+w_i)}\\
								   &\qquad-\frac{(d_1+d_2+d_3)^\alpha u_{d_1+d_2+d_3}}{\sum_{i=1}^\Delta i^\alpha u_i}\Bigg)  \binom{d_1+d_2+d_3}{d_1, d_2, d_3}  \prod_{s'=1}^{3} \bc{\frac{\eul^{-\beta \vecone\brk{s=s'}}}{2+\eul^{-\beta}}}^{d_{s'}}\\
	 =& 
	 - \frac{(d_1+d_2+d_3) \kappa u_{c,\theta}}{\sum_{i=1}^\Delta i(u_i+w_i)} -\frac{(d_1+d_2+d_3)^\alpha u_{c,\theta}}{\sum_{i=1}^\Delta i^\alpha u_i} 
	 \\&+ \Bigg(\frac{\kappa(d_1+d_2+d_3+1)(w_{d_1+d_2+d_3+1})}{\sum_{i=1}^\Delta i(u_i+w_i)}\Bigg)\frac {1} 9 \binom{d_1+d_2+d_3}{d_1, d_2, d_3}  \prod_{s'=1}^{3} \bc{\frac{\eul^{-\beta \vecone\brk{s=s'}}}{2+\eul^{-\beta}}}^{d_{s'}} \\
	 & + \Bigg(\frac{\kappa(d_1+d_2+d_3+1)(u_{d_1+d_2+d_3+1}/3)}{\sum_{i=1}^\Delta i(u_i+w_i)}\Bigg) \frac {1} 9 \binom{d_1+d_2+d_3}{d_1, d_2, d_3}  \prod_{s'=1}^{3} \bc{\frac{\eul^{-\beta \vecone\brk{s=s'}}}{2+\eul^{-\beta}}}^{d_{s'}}\\ 
	 = & 
	 - \frac{(d_1+d_2+d_3) \kappa u_{c,\theta}}{\sum_{i=1}^\Delta i(u_i+w_i)} -\frac{(d_1+d_2+d_3)^\alpha u_{c,\theta}}{\sum_{i=1}^\Delta i^\alpha u_i} 
	 + \sum_{s'=1}^3\frac{\kappa_{c,s',s}(d_{s'}+1)(w_{\theta^{+s'}})}{e_{s,s'}} 
	 + \sum_{s'=1}^3\frac{\kappa_{c,s',s}(d_{s'}+1)(u_{\theta^{+s'}})}{e_{s,s'}},
\end{align*}
as desired.
\end{proof}

\section{Proof of \Prop~\ref{prop_cleanup}}\label{sec_prop_cleanup}

\noindent
We prove \Prop~\ref{prop_cleanup} by showing that breadth-first-search from a random vertex of $\G^*_{\Delta,t}$ can be coupled with a sub-critical multi-type branching process.
One set of types of the branching process, corresponding to the vertices of the graph, are pairs $(s,\ell)$, where $s\in\{1,2,3\}$ represents the planted colour and $1\leq \ell\leq\Delta$ the total degree of the vertex.
The second set of types, corresponding to edges, are pairs $(s,s')$ of colours.
Each vertex type pair $(s,\ell)$ produces as offspring an expected number of $\ell\exp(-\beta\vecone\{s=s'\})/(2+\exp(-\beta))$ edges $(s,s')$ for every $s'\in\{1,2,3\}$.
Furthermore, the offspring of an edge type $(s,s')$ is a single $(s',\ell)$ vertex, where $\ell$ is drawn from the distribution $\ell(\vu_\ell+\vw_\ell)/(\sum_{i=1}^\Delta i(\vu_i+\vw_i))$.
Fact~\ref{prop_deferred}, Fact~\ref{prop_diffeq} and \Prop~\ref{prop_flat} ensure that this branching process models BFS on $\G^*_{\Delta,t}$ \whp

In analogy to the considerations in \Sec~\ref{sec_prop_flat}, we obtain the offspring matrix of the aforementioned branching process as follows.
Let
\begin{align}\label{eq_prop_cleanup_0}
	{M_{\mathrm{T}}^V}_{(s_2, s_3),(s_1, \ell)} = & \vecone\cbc{s_1 = s_2} \frac{(\ell-1) \eul^{-\beta\vecone\cbc{s_2 = s_3}}}{2+\eul^{-\beta}},&
	{M_{\mathrm{T}}^E}_{(s_3, \ell)(s_1, s_2)} = & \vecone\cbc{s_2 = s_3} \frac{(u_\ell + w_\ell)\ell}{u^{(1)}+w^{(1)}}.
\end{align}
Then the offspring matrix of the branching process reads $\begin{pmatrix}0&M_\mathrm{T}^V\\M_\mathrm{T}^E&0\end{pmatrix}$.
Hence, as in the proof of \Lem~\ref{lem_eigenvector} it suffices to show that under Assumption~\ref{assumption_ode} the matrix $M_{\mathrm T}=M_{\mathrm T}^VM_{\mathrm T}^E$ satisfies
\begin{align}\label{eq_prop_cleanup_1}
	\|M_{\mathrm T}\|<1.
\end{align}

Using~\eqref{eq_prop_cleanup_0}, we compute
\begin{align}\nonumber
	{M_{\mathrm{T}}}_{(s_3, s_4),(s_1, s_2)} 
	=&  \vecone\cbc{s_2=s_3} \frac{\eul^{-\beta\vecone\cbc {s_3=s_4}}}{(2+\eul^{-\beta})}\sum_{\ell=1}^{\infty} \frac{(u_\ell+w_\ell) \ell (\ell-1)} {u^{(1)}+w^{(1)}}\\
	=& \vecone\cbc{s_2=s_3} \frac{\eul^{-\beta\vecone\cbc {s_3=s_4}}}{(2+\eul^{-\beta}) (u^{(1)}+w^{(1)})} \sum_{\ell=1}^{\infty} (u_\ell+w_\ell) \ell (\ell-2).
\end{align}
Further, recalling \eqref{eq_prop_flat_1}, we obtain
\begin{align}
	\nonumber
	\|M_{\mathrm{T}}\|&\leq \abs{\frac{\mom u 2 + \mom w 2 - \mom u 1 - \mom w 1 }{(2+\eul^{-\beta}) (u^{(1)}+w^{(1)})}}
\cdot\norm{ \id\tensor \bc{ 
		 \begin{matrix}
				\eul^{-\beta} & \eul^{-\beta} & \eul^{-\beta}\\
				1 & 1 & 1 \\
				1 & 1 & 1
		\end{matrix}} }\\
				  &=\frac{\mom u 2 + \mom w 2 - \mom u 1 - \mom w 1 }{\mom u 1 + \mom w 1}.\label{eq_prop_cleanup_2}
\end{align}
Finally, \eqref{eqODEend} ensures that the r.h.s.\ of~\eqref{eq_prop_cleanup_2} is smaller than one.
The assertion~\eqref{eqprop_cleanup} follows from the fact that under the flat-white solution from \Prop~\ref{prop_flat} the planted colours $\SIGMA^*(\nix)$ and the colour lists at the termination of $\algo$ are asymptotically independent.

\section{Proofs of the main results}\label{sec_proof_main}

\noindent
To complete the proofs of \Thm~\ref{thm_main} and \Cor~\ref{cor_main} we just need to remind ourselves of the following elementary characterisation of the posterior distribution $\pr\brk{\SIGMA^*=\nix\mid\G^*}$.
For a graph $G=(V,E)$ and $\sigma\in\{1,2,3\}^V$ let $\Min(G,\sigma)$ be the number of edges $vw\in E$ such that $\sigma(v)=\sigma(w)$ and let $\Mout(G,\sigma)=|E|-\Min(G,\sigma)$.
Moreover, let
\begin{align*}
	\Nin(\sigma)&=\sum_{i=1}^3\binom{|\sigma^{-1}(i)|}2,&
	\Nout(\sigma)&=\sum_{1\leq i<j\leq 3}|\sigma^{-1}(i)|\cdot|\sigma^{-1}(j)|
\end{align*}
be the number of potential monochromatic/bichromatic edges, respectively.

\begin{fact}\label{lem_posterior}
	For a colouring $\sigma:V_n\to\{1,2,3\}$ and a graph $G$ we have
	\begin{align*}
		\pr\brk{\SIGMA^*=\sigma\mid\G^*=G}&\propto\din^{\Min(G,\sigma)}\dout^{\Mout(G,\sigma)}\bc{1-\frac\din n}^{\Nin(\sigma)-\Min(G,\sigma)}\bc{1-\frac\dout n}^{\Nout(\sigma)-\Mout(G,\sigma)}.
	\end{align*}
\end{fact}
\begin{proof}
	This is a routine calculation based on Bayes' formula that we detail for the sake of completeness.
	Recalling the definition~\eqref{eqdindout}--\eqref{eqpvw} of the \SBM, we obtain
	\begin{align*}
		\pr\brk{\SIGMA^*=\sigma\mid\G^*=G}&=\frac{\pr\brk{\G^*=G\mid\SIGMA^*=\sigma}\pr\brk{\SIGMA^*=\sigma}}{\pr\brk{\G^*=G}}\\
										  &\propto\din^{\Min(G,\sigma)}\dout^{\Mout(G,\sigma)}\bc{1-\frac\din n}^{\Nin(\sigma)-\Min(G,\sigma)}\bc{1-\frac\dout n}^{\Nout(\sigma)-\Mout(G,\sigma)},
	\end{align*}
	as claimed.
\end{proof}

\begin{proof}[Proof of \Thm~\ref{thm_main}]
	Let $\SIGMA^\dagger$ be the colouring from \Cor~\ref{cor_overlap}.
	Then \eqref{eq_cor_overlap} readily shows that
	\begin{align}\label{eq_thm_main_0}
		\agree(\SIGMA^\dagger)=o(1)
	\end{align}
	\whp\
	Furthermore, \Prop s~\ref{prop_high}, \ref{prop_condex} and~\ref{prop_cleanup} show that \whp
	\begin{align}\label{eq_thm_main_1}
		\Min(\G^*,\SIGMA^\dagger)&=o(n)&\mbox{and hence}&&
		\Mout(\G^*,\SIGMA^\dagger)&\sim dn/2.
	\end{align}
	Indeed, \Prop~\ref{prop_high} shows that the number of monochromatic edges $vw$ with $v,w\in V^\#$ is $o(n)$ \whp\
	Moreover, \Prop~\ref{prop_condex} states that \whp\ the number of bad vertices produced by $\algo$ is $o(n)$ \whp\ 
	Since all bad vertices are of bounded degree, this implies that the number of ensuing monochromatic edges is $o(n)$ \whp\
	Finally, \Prop~\ref{prop_cleanup} shows that \whp\ the number of monochromatic edges amongst the leftover vertices after $\algo$ finishes is $o(n)$ \whp\
	In addition, \eqref{eq_cor_overlap} implies that \whp\ we have
	\begin{align}\label{eq_thm_main_2}
		\Nin(\SIGMA^\dagger)&\sim3\binom{n/3}2\sim\frac13\binom n2,&
		\Nout(\SIGMA^\dagger)&\sim\frac23\binom n2.
	\end{align}

	Combining~\eqref{eq_thm_main_1}--\eqref{eq_thm_main_2} with Fact~\ref{lem_posterior}, we conclude that \whp
	\begin{align}\label{eq_thm_main_3}
		\pr\brk{\SIGMA^*=\SIGMA^\dagger\mid\G^*=G}&\propto\Theta\bc{\dout^{dn/2}}.
	\end{align}
	Since $\din<\dout$ for any $\beta>0$, \eqref{eq_thm_main_3} implies that \whp
	\begin{align}\label{eq_thm_main_5}
		\log\pr\brk{\SIGMA^*=\SIGMA^\dagger\mid\G^*=G}&\sim\max_{\sigma\in\{1,2,3\}^{V_n}}\log\pr\brk{\SIGMA^*=\sigma\mid\G^*=G}.
	\end{align}
	Finally, it is an immediate consequence of the definition \eqref{eqdindout}--\eqref{eqpvw} of the \SBM\ that \whp
	\begin{align}\label{eq_thm_main_4}
		\min(\G^*,\SIGMA^*)&\sim\frac {dn}{4\eul^\beta+2}.
	\end{align}
	Combining \eqref{eq_thm_main_3}--\eqref{eq_thm_main_4} with Fact~\ref{lem_posterior}, we conclude that \whp
	\begin{align}\label{eq_thm_main_6}
		L_{\G^*}(\SIGMA^*)=\frac{d\beta}{4\eul^{\beta}+2}.
	\end{align}
	Thus, the assertion follows from \eqref{eq_thm_main_0}, \eqref{eq_thm_main_5} and \eqref{eq_thm_main_6}.
\end{proof}

\begin{proof}[Proof of \Cor~\ref{cor_main}]
	As in the proof of \Thm~\ref{thm_main} we consider the colouring $\SIGMA^\dagger$ from \Cor~\ref{cor_overlap}, which satisfies \eqref{eq_thm_main_0} \whp\
	We claim that $\SIGMA^\dagger$ is indeed a proper 3-colouring of $\G^*$ with probability $\Omega(1)$.
	Indeed, \Prop~\ref{prop_high} shows that $\SIGMA^\dagger$ does indeed induce a proper 3-colouring on the subgraph $\G^\#$ with probability $\Omega(1)$.
	Additionally, we claim that $\algo$ does not produce any bad vertices with probability $\Omega(1)$.
	To see this, we revisit the proof of \Prop~\ref{prop_condex} from \Sec~\ref{sec_prop_condex}.
	The key element of that proof is the directed rooted graph $\fT_t$.
	If $\fT_t$ is acyclic, then no bad vertex is produced in the $t$-th epoch.
	Furthermore, the proof of \Lem~\ref{lem_kappa} shows that under Assumption~\ref{assumption_ode} $\fT_t$ can be coupled with a sub-critical branching process. 
	Thus, the expected progeny of the branching process is bounded, while under assumption~\eqref{eq_cond_prop_condex} a linear number of unexplored vertices remain.
	In effect, during each epoch there is a probability of $O(1/n)$ that $\fT_t$ contains a cycle.
	Moreover, due to Fact~\ref{prop_deferred} the emergence of cycles in different epochs are conditionally independent given the $\sigma$-algebras $(\fD_t)_t$.
	Therefore, with probability $\Omega(1)$ the directed graph $\fT_t$ is acyclic for all epochs $t$.
	Finally, the proof of \Prop~\ref{prop_cleanup} in \Sec~\ref{sec_prop_cleanup} again relies on the fact that the sub-graph of $\G^*_\Delta$ left upon termination of $\algo$ is a sub-critical random graph.
	Hence, this graph is acyclic with probability $\Omega(1)$, in which case $\SIGMA^\dagger$ is a proper 3-colouring.

	To complete the proof of \Cor~\ref{cor_main} assume that $\SIGMA^\dagger$ is indeed a proper 3-colouring of $\G^*$ that satisfies \eqref{eq_thm_main_0}.
	\Whp\ the random graph $\G^*$ contains $\Omega(n)$ isolated vertices, which by construction receive uniformly random colours.
	Furthermore, \eqref{eq_cor_overlap} implies that \whp\ the colour classes $\SIGMA^{\dagger\,-1}(i)$ of $\SIGMA^\dagger$ satisfy $|\SIGMA^{\dagger\,-1}(i)|\sim n/3$ for $i=1,2,3$.
	Therefore, by suitably recolouring the isolated vertices we can obtain a perfectly balanced proper 3-colouring $\SIGMA^\ddagger$ of $\G^*$.
	Since Fact~\ref{lem_posterior} shows that a balanced proper 3-colouring maximises the posterior probability $\pr\brk{\SIGMA^*=\nix\mid\G^*}$, the assertion follows.
\end{proof}

\subsubsection*{Acknowledgement.}
The first author thanks Dimitris Achlioptas and Dimitris Papailiopoulos for helpful discussions that inspired this work.
Amin Coja-Oghlan and Lena Krieg are supported by DFG CO 646/3 and DFG CO 646/5.  Olga Scheftelowitsch is supported by DFG CO 646/5.

\begin{appendix}
\section{Numerics and Simulations}\label{apx:num_res}

\noindent
In this section we present some computational evidence in support of Assumption~\ref{assumption_ode}. 
One set of results is based on numerical solutions of the differential equations \eqref{eqODEinitial}--\eqref{eqkappa}.
Additionally, we implemented the $\algo$ algorithm, ran it on \sbm\ instances and compared the empirical values of the quantities on the l.h.s.\ of~\eqref{eqODEend}--\eqref{eqODElambda} with the values predicted by the numerical solutions to the ODEs.

\subsection{Numerical solution of the ODEs}\label{apx_num}
We solved the ODEs \eqref{eqODEend}--\eqref{eqODElambda} numerically by means of Matlabs {\tt ode45} function, which relies on the Runge-Kutta method.
However, this approach did not work out of the box, i.e., it failed to deliver plausible solutions to the ODEs.
This failure has presumably be attributed to the fact that some of the functions involved, particularly the functions $u_\ell,w_\ell$ for larger values of $\ell$, take tiny numerical values, while others are much bigger.
In effect, a naive invocation of the ODE solver produced implausible negative values for some values that should be strictly positive because of their underlying combinatorial interpretation.
To keep the solver on track, we therefore truncated the functions at zero.
In other words, in order to obtain plausible numerical solutions to the ODEs we replaced the functions $u_\ell,w_\ell$ by $\max(u_\ell,0)$, $\max(w_\ell,0)$. 
The numerical values of 
\begin{align*}
	d_{\max}(\alpha) &= \max\cbc{d>0: \max_{t:\gamma(t)>0}\lambda(t)<1}, \qquad\mbox{where}\\
	\lambda(t)&=
		\frac{2\sum_{i=1}^\Delta i(i-1)u_i(t)}{3\sum_{i=1}^\Delta i(w_i(t)+u_i(t))},\\
	\gamma(t)& = \sum_{\ell=0}^\Delta\ell(\ell-2)(u_\ell(t)+w_\ell(t))
\end{align*}
displayed in Figure~\ref{fig:dalpha}, are in perfect agreement with the values reported in~\cite{AchMoore3col}.
For example, for $\alpha=14$ we obtain $d=4.03$.

\begin{figure}
	\includegraphics[width=0.6\columnwidth]{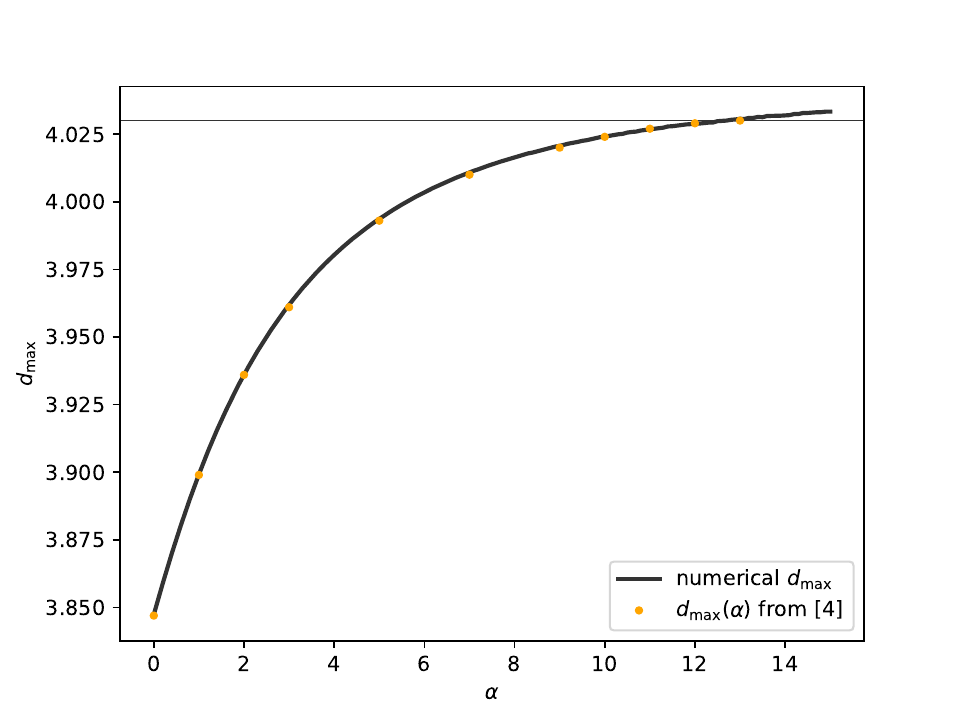}
	\caption{
		The black curve shows the maximum $d_{\max}=d_{\max}(\alpha)$ such that the eigenvalue $\lambda(t)$ remains bounded below according to the numerical solution of the ODEs.
		The orange dots display the corresponding values reported in \cite{AchMoore3col}.
		The horizontal line is placed at $d=4.03$.
	}
	\label{fig:dalpha}
\end{figure}

\subsection{Experiments with the $\algo$ algorithm}\label{apx_ex}

\begin{figure}[h]
	\hspace{0.3cm}
	\begin{subfigure}{0.49\textwidth}
		\includegraphics[width=1\columnwidth]{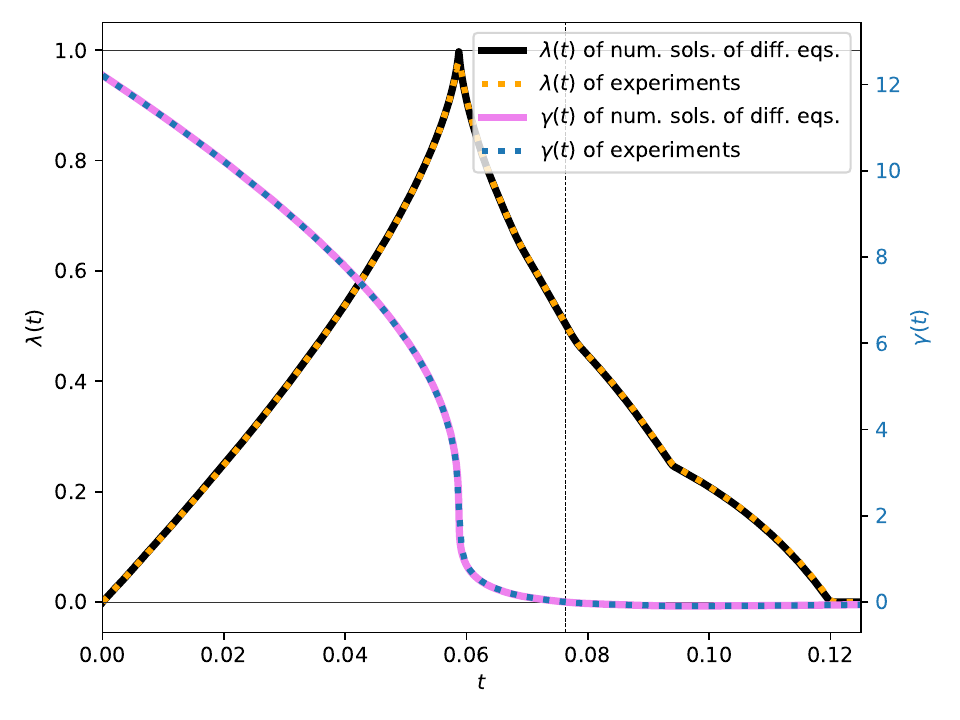}
	\end{subfigure} 
	\hspace{-0.3cm}
	\begin{subfigure}{0.49\textwidth}
		\includegraphics[width=1\columnwidth]{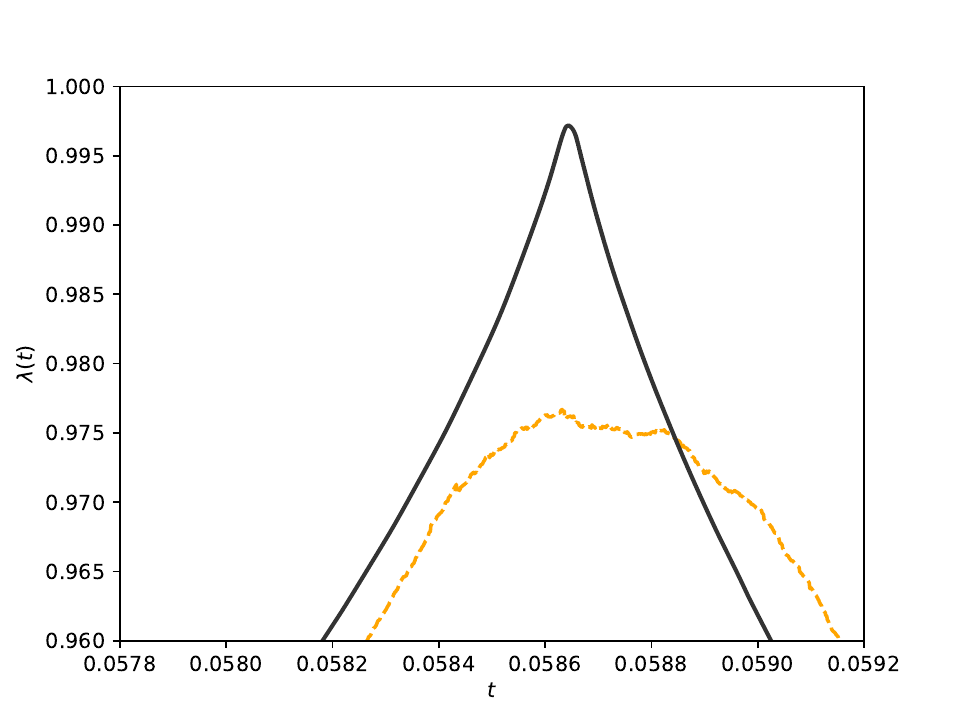}
	\end{subfigure}
	\caption{
		{\em Left:}	The eigenvalue $\lambda(t)$ at each time $t$ for $\alpha = 15$, $d=4.03$ and $\beta = 6$, i.e., slightly above the Kesten-Stigum threshold.
		The black curve displays the numerical solutions to the ODEs.
		The orange dotted curve is the empirical value of $\lambda(t)$ averaged over $10$ runs with $n=10^{6}$.
		The magenta line displays the numerical ODE value of $\gamma(t)$, while the blue dotted shows the experimental value. 
		The dashed vertical line marks the last point $t^*$ where \eqref{eqODEend} holds.
		{\em Right:} the black/orange lines from the left figure zoomed in on the peak.
	}\label{fig:lambda}
\end{figure}
\begin{figure}[h]
	\includegraphics[width=0.6\columnwidth]{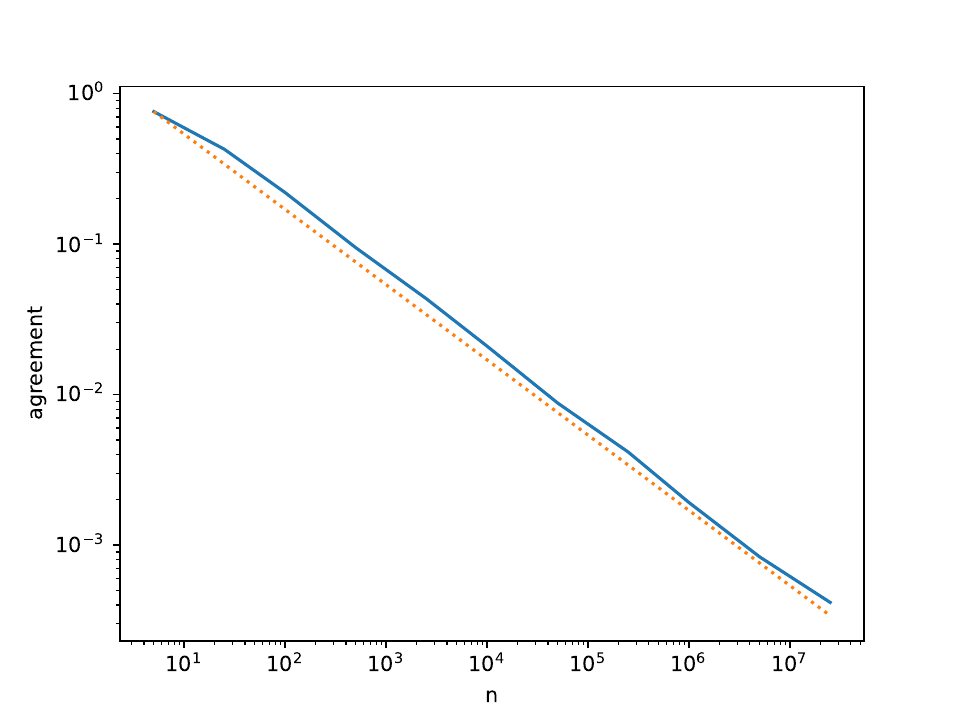}
	\caption{The blue line shows the average empirical agreement (as defined in~\eqref{eqagree}) over 20 runs of $\algo$ with $\beta=6$, $d=4.03$ and $\alpha=15$ on a $(\log,\log)$-scale. 
		The orange dotted line is just a straight line with slope $-1/2$.
		This is the expected gradient of the blue line, since the agreement of two independent random colourings is of order  $n^{-1/2}$.
	} 
	\label{fig:overlap}
\end{figure}

Partly in order to check the validity of the truncation discussed in Appendix~\ref{apx_num} and partly to get an idea of the rate of convergence, we implemented the $\algo$ algorithm and ran experiments on the \sbm\ for various values of $\alpha,\beta$ and $d$.
In particular, we tracked the experimental values of $\lambda(t)$ and $\gamma(t)$ with $u_\ell,w_\ell$ taking their (random) empirical values.

By comparison to the theoretical version of the algorithm (Algorithm~\ref{alg_AM}), the implementation comes with two modifications.
First, we do not truncate the high degree vertices.
Remember that the main purpose of this truncation was to facilitate the method of differential equations, while from a purely algorithmic viewpoint a separate treatment of high-degree vertices should not be necessary.
Second, we ran the algorithm until all vertices got coloured.
To this end the algorithm as implemented just picks a random vertex of list size three and assigns it a random colour whenever the algorithm runs out of vertices of list size one or two.

To our moderate surprise, even for moderate numbers $n$ of vertices the experimental values of $\lambda(t),\gamma(t)$ emerged to be in agreement with the values predicted by the numerical solution to the ODEs, to the extent that the curves are difficult to tell apart; see Figure~\ref{fig:lambda}.
We produced this figure as follows.
We compared the maximum eigenvalue $\lambda(t)$ predicted by the ODEs with the experimental value.
The left side of figure \ref{fig:lambda} displays this comparison: the values seem to align perfectly.
In particular, the largest eigenvalue is always smaller than $1$, both numerically and experimentally.
At first glance the eigenvalue seems to have a non differentiable point in the middle.
Yet zooming in shows that the peak is indeed a smooth curve.
In this close up we can also see a small discrepancy between the experimental (orange) and numerical (black) values.


Additionally, we ran the implementation of $\algo$ to investigate the overlap between the colouring $\sigma$ produced by the algorithm and the ground truth $\SIGMA^*$, as well as the number of bad vertices.
Figure~\ref{fig:overlap} shows the agreement $\agree(\sigma)$ for different values of $n$.
As one would expect if $\sigma$ and $\SIGMA^*$ are actually uncorrelated, the empirical agreement scales as $\Theta(n^{-1/2})$.
Finally, Figure~\ref{fig:badvertices} displays the number of bad vertices produced by $\algo$ along with the frequency of experiments that produced no bad vertices at all.



\begin{figure}
	\begin{subfigure}{0.49\textwidth}
		\includegraphics[width=1\columnwidth]{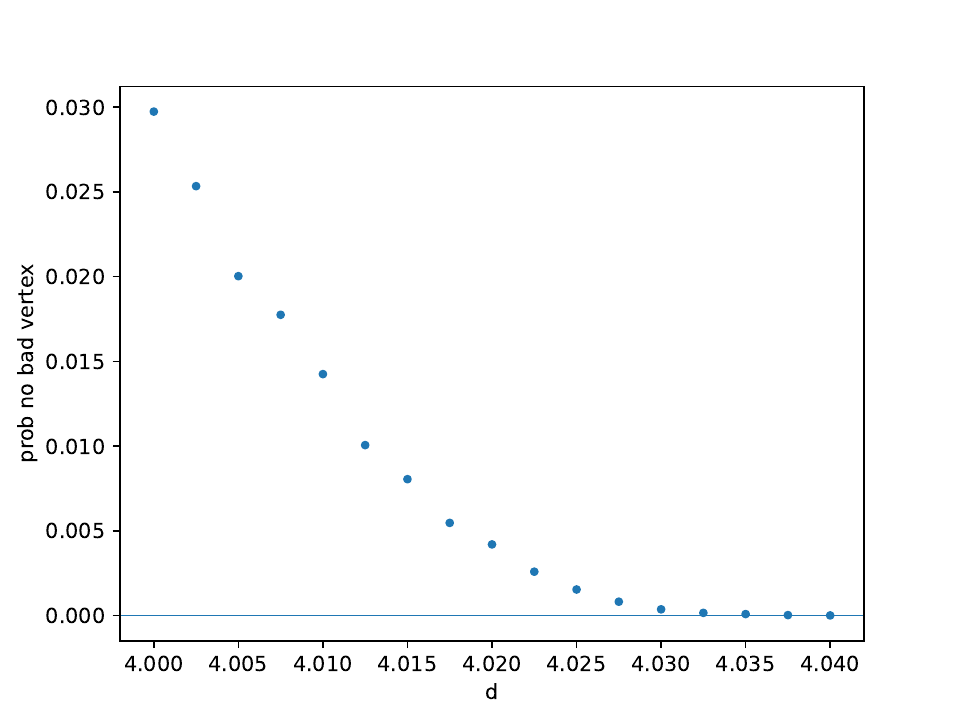}
		\caption{Fraction of experiments with \emph{no} bad vertices, $\algo$ found valid colouring}
		\label{fig:badvertices_prob}
	\end{subfigure} 
	\begin{subfigure}{0.49\textwidth}
		\includegraphics[width=1\columnwidth]{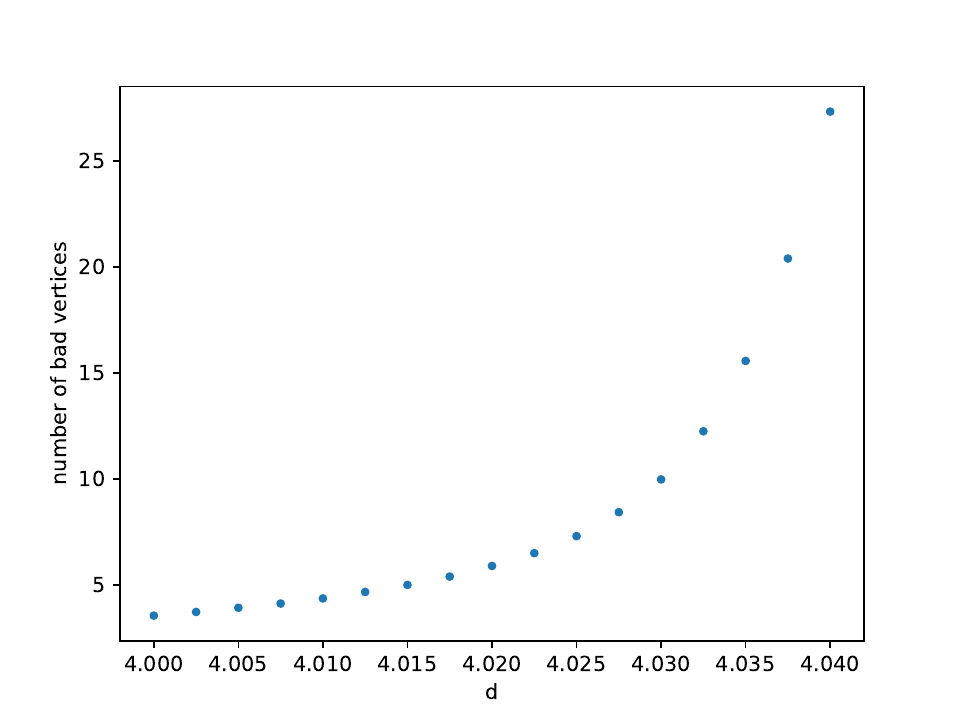}
		\caption{Average number of bad vertices}
	\end{subfigure}
	\caption{
		Bad vertices in $10^5$ simulated runs with $n=10^6$, $\alpha=15$ and $\beta=6$.	}
	\label{fig:badvertices}
\end{figure}

\end{appendix}

\end{document}